\documentclass[11pt]{article}
\usepackage[numbers,sort&compress]{natbib}
\usepackage{enumerate}
\usepackage{amscd}
\usepackage{amsmath}
\usepackage{latexsym}
\usepackage{amsfonts}
\usepackage{setspace}
\usepackage{amssymb}
\usepackage{amsthm}
\usepackage{verbatim}
\usepackage{mathrsfs}
\usepackage{enumerate}
\usepackage[hypertexnames=false]{hyperref}

\theoremstyle{plain}
\theoremstyle{definition}\newtheorem{theorem}{Theorem}[section]
\theoremstyle{definition}\newtheorem{lemma}[theorem]{Lemma}
\theoremstyle{plain}
\theoremstyle{plain}\newtheorem{prop}[theorem]{Proposition}
\theoremstyle{definition}\newtheorem{remark}{Remark}[section]
\usepackage{xcolor}

\newcommand{\mr}{\mathbb{R}}
\newcommand{\mt}{\mathbb{T}}
\newcommand{\mz}{\mathbb{Z}}

\newcommand{\ben}{\begin{enumerate}}
	\newcommand{\een}{\end{enumerate}}

\newcommand{\dd}{\mathrm{~d}}
\newcommand{\w}{\widehat}
\newcommand{\tb}{\tilde{b} }

\newcommand{\Rmnum}[1]{\expandafter\@slowromancap\romannumeral #1@}

\allowdisplaybreaks

\numberwithin{equation}{section}
\begin{document}

	\title{Global stability and asymptotic behavior for the incompressible MHD equations without viscosity or magnetic diffusion}
\author{Qunyi Bie\,\,
Hui Fang\,\,
Yanping Zhou
}
\date{}
\maketitle
\begin{abstract}
Physical experiments and numerical simulations have revealed a remarkable stabilizing phenomenon: a background magnetic field stabilizes and dampens electrically conducting fluids. This paper provides a rigorous mathematical justification of this effect for the $n$-dimensional incompressible magnetohydrodynamic equations with partial diffusion on periodic domains. We establish the global stability and derive explicit decay rates for perturbations around an equilibrium magnetic field satisfying the Diophantine condition. Our results yield the \textit{effective decay rates in all intermediate Sobolev norms} and \textit{significantly relax the regularity requirements} on the initial data compared with previous works (\textit{Sci. China Math.} 41:1--10, 2022; \textit{J. Differ. Equ.} 374:267--278, 2023; \textit{Calc. Var. Partial Differ. Equ.} 63:191, 2024). Furthermore, the analytical framework developed here is dimension-independent and can be flexibly adapted to other fluid models with partial dissipation.
\end{abstract}
\noindent {\bf MSC(2020):}\quad 35A01, 35B35, 35Q35, 76E25, 76W05.
    \vskip 0.02cm
	\noindent {\bf Keywords:} Incompressible MHD equations, stability, decay estimates.


\newpage
{\small 
\tableofcontents
}

\vskip 0.3cm

\section{Introduction and main results}
\subsection{Model and related research}
The magnetohydrodynamics (MHD) equations describe the motion of electrically conducting fluids, such as plasmas, liquid metals, and electrolytes, and are fundamental in geophysics, astrophysics, and engineering (see, e.g., \cite{Biskamp1993,Priest-Forbes2000,Davidson2001,duvaut-1972}).
Mathematically, they share essential features with the Euler and Navier-Stokes equations, while exhibiting additional structures that give rise to rich phenomena and analytical challenges. In this paper, we consider the incompressible MHD system on the $n$-dimensional torus $\mathbb{T}^n$ ($n\ge2$):
\begin{align}\label{mhd}
\left\{
\begin{array}{l}
\partial_t u - \mu \Delta u + u \cdot \nabla u + \nabla p = b \cdot \nabla b, \\[1mm]
\partial_t b - \nu \Delta b + u \cdot \nabla b = b \cdot \nabla u, \\[1mm]
\nabla \cdot u = \nabla \cdot b = 0, \\[1mm]
u(x,0) = u_0(x), \quad b(x,0) = b_0(x),
\end{array}
\right.
\end{align}
where $u$ and $b$ denote the velocity and magnetic field, respectively, $p$ is the pressure, and $\mu,\nu\ge0$ are the viscosity and magnetic diffusion coefficients.
Since dissipative mechanisms fundamentally influence the long-time behavior of solutions, much of the mathematical literature on \eqref{mhd} has been devoted to the well-posedness problem under different dissipative regimes, as briefly reviewed below.

\medskip
\noindent\textbf{(i) Full dissipation:} $\mu>0$, $\nu>0$.  The existence and uniqueness results for both weak and strong solutions in 2D could be found in Duvaut and Lions \cite{duvaut-1972} (see also  \cite{sermange-1983,Abidi-Paicu2008}). However, the corresponding problem in 3D remains open.

\medskip
\noindent\textbf{(ii) Only magnetic diffusion:} $\mu=0$, $\nu>0$.
Lei and Zhou \cite{Lei-Zhou2009} established the existence of global $H^1(\mathbb{R}^2)$ weak solutions via compactness arguments (see also \cite{Cao-Wu2011}), while the smoothness and uniqueness of such weak solutions remain open.  A related but weaker result was later obtained by Jiu, Niu, Wu, Xu, and Yu \cite{Jiu-Niu-Wu-Xu-Yu2015}. When the initial datum $(u_0, b_0)$ is assumed to be small, the first global well-posedness result is established by Wei and Zhang \cite{zhang-2020-GlobalWellPosedness2D} in $H^4(\mathbb{T}^2)$, assuming $\int_{\mathbb{T}^2} b_0\,dx = 0$. Ye and Yin \cite{Ye-Yin2022} later reduced the regularity requirement to $H^s(\mathbb{T}^2)$ with $s>2$. Nevertheless, the solutions in \cite{zhang-2020-GlobalWellPosedness2D, Ye-Yin2022} may grow in time, and in particular $\|\nabla u\|_{L^\infty}$ can exhibit exponential growth. Consequently, the stability problem near the trivial solution remains open even in the periodic setting.

\medskip
\noindent\textbf{(iii) Only viscosity:} $\mu>0$, $\nu=0$.
It remains an open problem whether global classical solutions exist even in 2D for generic smooth initial data. Chemin, McCormick,  Robinson and Rodrigo
 \cite{Chemin2016} established local well-posedness in critical Besov spaces.

\medskip
\noindent\textbf{(iv) Vanishing or small dissipation:} $\mu=0$, $\nu=0$.
For the fully non-dissipative system, it is largely unknown whether classical solutions can develop finite-time singularities, even in 2D. Existing global well-posedness results are available only near nontrivial background magnetic fields (see \cite{Bardos-Sulem1988,Frisch1983,Kraichnan1965,He-Xu-Yu2018,
Cai-Lei2018,Wei-Zhang2017,Wei-Zhang2018}). Bardos, Sulem, and Sulem \cite{Bardos-Sulem1988} first proved global well-posedness for initial data close to a nontrivial equilibrium, observing that a sufficiently strong magnetic field suppresses nonlinear interactions and prevents the formation of strong gradients \cite{Frisch1983,Kraichnan1965}. Moreover, He, Xu, and Yu \cite{He-Xu-Yu2018} (see also \cite{Cai-Lei2018,Wei-Zhang2017}) justified the vanishing viscosity limit of the fully diffusive MHD system to the solutions constructed by \cite{Bardos-Sulem1988} under certain structural relations between $\mu$ and $\nu$.

\medskip
The above theoretical results are consistent with experimental and numerical evidence, both indicating that a sufficiently strong magnetic field tends to stabilize electrically conducting fluids (see, e.g., \cite{HAlfvn1942,Alemany-Frisch1979,Davidson1995,Davidson1997,Alexakis2011,Gallet-Mordant2009,Gallet-Doering2015}). Motivated by this phenomenon, the viscous but non-resistive case ($\mu>0,\nu=0$) attracted particular attention. An early conjecture suggested that a non-vanishing background field could dissipate energy at a rate independent of ohmic resistivity \cite{Califano19999,Majda1984}. Subsequently, rigorous results on global well-posedness were obtained for perturbations around strong background magnetic fields (see, e.g., \cite{lin-2014-GlobalSmallSolutions,Lin-Zhang2015simplified-proof,Lin-Xu-Zhang2015JDE,ren2014global,zhangTJDE2016,xu-2015-GlobalSmallSolutions,Abidi-Zhang2017,Pan-Zhou-Zhu2018,Deng2018zhang,jiang2021}). Parallel progress was achieved for the complementary non-viscous but resistive case ($\mu=0,\nu>0$) (see, e.g., \cite{zhou-2018-GlobalClassicalSolutions,Liu-Zhang2025Linear,Lin-Suo-Wu2025}).
Among these works, Deng and Zhang \cite{Deng2018zhang} derived explicit decay estimates for the MHD system in $\mathbb{R}^3$ with a strong background field $\bar{b}=(0,0,1)^T$ within the Lagrangian framework, revealing a clear dissipative mechanism. This naturally raises the question of whether a similar stabilization effect can be understood directly in Eulerian coordinates, which offer a simpler structure and avoid derivative loss.

In fact, beyond the strong background field $\bar{b}$, another important candidate is a constant magnetic field $\tilde{b}$ satisfying the so-called {\it Diophantine condition}, which has recently attracted considerable attention (see \cite{chen-2022-3dmhd-Diophant,zhai-2023-2dmhdstability-Diophant,XieJiu-CVPDE2024,Jiu-Liu-Xie-compressibleMHD,WuZhai,lizhai-2023-jga,Wang-Zhang-Zhai2024,Zhai-Wu-Xu2025,Dai-Lai-Zhai2025,jiang-2023-arma-nonresistive}).
\newline{\it Diophantine condition}: For every nonzero vector $k \in \mathbb{Z}^n$ and $r>n-1$, there exists a constant $c>0$ such that
\begin{align}\label{Diophantine}
	|\tilde{b}\cdot k|\geq \frac{c}{|k|^r}.
\end{align}
To see why this condition naturally arises, let us recall the proof of Deng and Zhang \cite{Deng2018zhang}, which was conducted in the Lagrangian framework. In Lagrangian coordinates, the perturbation system takes the form
\[
Y_{tt} - \Delta_y Y_t - \partial_{b_0}^2 Y = F(Y),
\]
with $\partial_{b_0} = b_0 \cdot \nabla_y$. A crucial step in their analysis was to introduce a coordinate transformation so that $\partial_{b_0}$ becomes $\tilde{b}\cdot\nabla$ when $b_0 - \tilde{b}$ is small and compactly supported. In the Eulerian setting on the torus $\mathbb{T}^n$, a similar straightening can be performed directly only if $\tilde{b}$ satisfies the {\it Diophantine condition} (see \cite{Alinhac}, p.~137), providing a natural motivation to study this class of background fields in the periodic setting. Mathematically, the {\it Diophantine condition} is satisfied for almost all $\tilde{b}\in\mathbb{R}^n$, while it fails if the components of $\tilde{b}$ are rational or one component vanishes \cite{chen-2022-3dmhd-Diophant}.
When the magnetic field $b$ is close to $\tilde{b}$, it suffices to consider the following perturbation system by writing $b$ for $b - \tilde{b}$:
\begin{align}\label{equation}
	\left\{
	\begin{array}{l}
		\partial_t u-\mu\Delta u+u \cdot \nabla u+\nabla p=\tilde{b} \cdot \nabla b+b \cdot \nabla b, \\[1mm]
		\partial_t b-\nu\Delta b+u \cdot \nabla b=\tilde{b} \cdot \nabla u+b \cdot \nabla u, \\[1mm]
		\nabla\cdot u=\nabla \cdot b=0, \\[1mm]
		u(x, 0)=u_0(x), \quad b(x, 0)=b_0(x).
	\end{array}
	\right.
\end{align}
For system \eqref{equation} with $\mu=1, \nu=0$ or $\mu=0, \nu=1$, Chen, Zhang, and Zhou \cite{chen-2022-3dmhd-Diophant} first studied stability in the 3D periodic domain $\mathbb{T}^3$, proving asymptotic stability in $H^{4r+7}(\mathbb{T}^3)$ for $r>n-1$,  which was then improved to $H^{(3 + 2\beta)r + 5 + (\alpha + 2\beta)}(\mathbb{T}^2)$ for 2D periodic domain $\mathbb{T}^2$ and any $\alpha > 0$, $\beta > 0$ by Zhai \cite{zhai-2023-2dmhdstability-Diophant}. Recently, Xie, Jiu, and Liu \cite{XieJiu-CVPDE2024} exploited a new dissipative mechanism to significantly lower the regularity requirement compared with the earlier works \cite{chen-2022-3dmhd-Diophant,zhai-2023-2dmhdstability-Diophant}, establishing stability in $H^{(3r+3)^+}(\mathbb{T}^n)$ for general $n$, addressing both linear and nonlinear stability with consistent decay rates. For brevity, we only recall the nonlinear result here.
\newline{\bf Theorem (Xie-Jiu-Liu,\,\cite{XieJiu-CVPDE2024})}
{\it Let $n\geq 2$, $r>n-1$ and $\mu=1$, $\nu=0$ or $\mu=0$, $\nu=1$. Assume that $(u_0,b_0)\in H^m(\mt^n)$ for $m>3(r+1)$ satisfying
\begin{align*}
	\int_{\mathbb{T}^n}u_0\,\dd x=\int_{\mathbb{T}^n}b_0\,\dd x=0,
\end{align*}
and there is a small constant $\varepsilon$ so that
\begin{align*}
	\|u_0\|_{H^m(\mt^n)}+\|b_0\|_{H^m(\mt^n)}\leq \varepsilon.
\end{align*}
Then there exists a unique and global solution $(u,b)\in C([0,\infty); H^m(\mt^n))$ of incompressible MHD equations \eqref{equation} such that for any $t\geq 0$,
\begin{align*}
	\|u(t)\|_{H^m(\mt^n)}+\|b(t)\|_{H^m(\mt^n)}\leq C\varepsilon,
\end{align*}
and
\begin{align}\label{2024finaldecay}
	\|u(t)\|_{H^s(\mt^n)}+\|b(t)\|_{H^s(\mt^n)}\leq C(1+t)^{-\frac{m-s}{2(r+1)}},\quad{\rm for\,\,\,any}\,\,r+1\leq s\leq m.
\end{align}}
Later on, Jiu, Liu and Xie \cite{Jiu-Liu-Xie-compressibleMHD} extended this framework to the compressible isentropic MHD equations without magnetic diffusion, obtaining analogous stability results in $H^{(3r+3)^+}(\mathbb{T}^n)$.
Further developments along this direction can be found in the studies of compressible MHD \cite{WuZhai,lizhai-2023-jga} and other related fluid models with Diophantine background fields \cite{Wang-Zhang-Zhai2024,Zhai-Wu-Xu2025,Dai-Lai-Zhai2025}.

Building on the above works, two natural questions arise:
\begin{itemize}
    \item Can the Sobolev regularity threshold on the initial data, namely $m>3(r+1)$ in \cite{XieJiu-CVPDE2024}, be further lowered while still ensuring global stability?
    \item Can the decay estimate \eqref{2024finaldecay}, which requires $s\geq r+1$, be further relaxed?
\end{itemize}
In this paper, we continue this line of research for system \eqref{equation} with $\mu>0$, $\nu=0$ or $\mu=0$, $\nu>0$, under a background magnetic field satisfying the {\it Diophantine condition}.
Our main novelty is twofold: first, we {\it greatly reduce} the Sobolev regularity required on the initial data for stability and decay; second, we establish decay rates of $(u,b)$ in $H^s$ for all $s\in[0,m]$, building on the idea of \cite{Elgindi2010,Jiang-Kim2023CVPDE}.
A discussion of the mechanisms behind these results is postponed to Section~\ref{mechanisms}.

\subsection{Major achievements}\label{Major}
We now present the main results of this paper. The first two theorems concern the linearized system of \eqref{equation} in the velocity diffusive case ($\mu=1, \nu=0$) and the magnetic diffusive case ($\mu=0, \nu=1$), respectively. We subsequently establish global existence and decay estimates for the nonlinear system.
\begin{theorem}[Linear stability I]\label{thm1}
Let \( n \geq 2 \), \(r > n-1\), and $\mu=1$, $\nu=0$.
Suppose that \( \tilde{b} \) satisfies the Diophantine condition, and that \( (U_0, B_0) \in H^m(\mathbb{T}^n) \) for some \( m \geq 0 \), with
\begin{align}\label{mean-zero}
	\int_{\mathbb{T}^n} U_0\, \mathrm{d}x = 0, \qquad
    \int_{\mathbb{T}^n} B_0\, \mathrm{d}x = 0.
\end{align}
Let \( (U, B) \) be the corresponding solution to the linear system
\begin{align}\label{lineartheorem}
	\left\{\begin{array}{l}
		\partial_t U-\Delta U=\tilde{b} \cdot \nabla B, \\[0.5ex]
		\partial_t B=\tilde{b} \cdot \nabla U, \\[0.5ex]
        \nabla\cdot U=\nabla \cdot B=0, \\[0.5ex]
		U(x, 0)=U_0(x), \quad B(x, 0)=B_0(x).
	\end{array}\right.
\end{align}
Then, for any \( s \in [0, m] \), the following decay estimates hold:
\begin{align*}
	\lVert U(t) \rVert_{H^s} &\leq C(1 + t)^{-\left( \frac{1}{2} + \frac{m - s + 1}{2(1 + r)} \right)} \lVert (U_0, B_0) \rVert_{H^m}, \\[1ex]
	\lVert B(t) \rVert_{H^s} &\leq C(1 + t)^{-\frac{m - s}{2(1 + r)}} \lVert (U_0, B_0) \rVert_{H^m}.
\end{align*}
\end{theorem}
\begin{theorem}[Linear stability II]\label{thm1.1}
Let \( n \geq 2 \), \(r > n-1\), and $\mu=0$, $\nu=1$.
Suppose that \( \tilde{b} \) satisfies the Diophantine condition, and that \( (U_0, B_0) \in H^m(\mathbb{T}^n) \) for some \( m \geq 0 \), with
\[
	\int_{\mathbb{T}^n} U_0\, \mathrm{d}x = 0, \qquad
    \int_{\mathbb{T}^n} B_0\, \mathrm{d}x = 0.
\]
Let \( (U, B) \) be the corresponding solution to the linear system
\begin{align}\label{lineartheorem2}
	\left\{\begin{array}{l}
		\partial_t U=\tilde{b} \cdot \nabla B, \\[0.5ex]
		\partial_t B-\Delta B=\tilde{b} \cdot \nabla U, \\[0.5ex]
		\nabla\cdot U=\nabla \cdot B=0, \\[0.5ex]
		U(x, 0)=U_0(x), \quad B(x, 0)=B_0(x).
	\end{array}\right.
\end{align}
Then, for any \( s \in [0, m] \), the following decay estimates hold:
\begin{align*}
\lVert U(t) \rVert_{H^s} &\leq C(1 + t)^{-\frac{m - s}{2(1 + r)}} \lVert (U_0, B_0) \rVert_{H^m}, \\[1ex]
	\lVert B(t) \rVert_{H^s} &\leq C(1 + t)^{-\left( \frac{1}{2} + \frac{m - s + 1}{2(1 + r)} \right)} \lVert (U_0, B_0) \rVert_{H^m}.
\end{align*}
\end{theorem}
\begin{remark}
Under the same zero-mean assumption, Xie, Jiu, and Liu \cite{XieJiu-CVPDE2024} proved that in the presence of either kinematic viscosity or magnetic diffusion,
\[
\|(U,B)(t)\|_{H^s(\mathbb{T}^n)}
   \leq C (1+t)^{-\frac{m-s}{2(1+r)}}
   \|(U_0,B_0)\|_{H^m(\mathbb{T}^n)},
   \quad 0 \leq s \leq m.
\]
Here, Theorem \ref{thm1} (with $\mu=1$, $\nu=0$) improves the decay of $U$ to
\[
\|U(t)\|_{H^s} \leq C (1+t)^{-\left(\tfrac{1}{2}+\tfrac{m-s+1}{2(1+r)}\right)} \|(U_0,B_0)\|_{H^m},
\]
while Theorem \ref{thm1.1} (with $\mu=0$, $\nu=1$) yields the same sharper rate for $B$. Therefore, in both degenerate cases, our results provide a faster algebraic decay than that of \cite{XieJiu-CVPDE2024}.
\end{remark}

Our next two theorems address the nonlinear MHD equations, establishing global well-posedness and time-decay estimates for both the magnetic diffusive case $(\mu,\nu)=(0,1)$ and the velocity diffusive case $(\mu,\nu)=(1,0)$. The regularity thresholds on the initial data are stated below, with a brief discussion of their differences postponed to Section \ref{mechanisms}.
\begin{theorem}[Nonlinear stability I]\label{thm.2}
	Let $n\geq 2$, $r>n-1$ and $\mu=0$, $\nu=1$. Let \(m \in \mathbb{N}\) satisfy
\begin{align}\label{m}
m > 4 + 2r + \frac{n}{2}.
\end{align}
Assume that \(\tilde{b}\) satisfies the Diophantine condition, and that the initial data \((u_0, b_0) \in H^m(\mathbb{T}^n)\) satisfy
\[
\mathrm{div}\, u_0 = \mathrm{div}\, b_0 = 0, \quad
\int_{\mathbb{T}^n} u_0\, \mathrm{d}x = \int_{\mathbb{T}^n} b_0\, \mathrm{d}x = 0,
\]
together with the smallness condition
\begin{align}\label{smallcondition}
	\|u_0\|_{H^m(\mt^n)}+\|b_0\|_{H^m(\mt^n)}\leq \varepsilon,
\end{align}
for some sufficiently small constant \(\varepsilon > 0\).
Then the incompressible MHD equations \eqref{equation} admit a unique global solution \((u, b)\) satisfying
\begin{align}\label{eq:1.2a}
\sup_{t \in [0, \infty)} \lVert (u, b)(t) \rVert_{H^m}+ \int_0^\infty \lVert \nabla b(t) \rVert_{H^m}^2 \,{d}t \leq C \lVert (u_0, b_0) \rVert_{H^m},
\end{align}
and the decay estimate
\begin{align}\label{finaldecay111}
\lVert (u, b)(t) \rVert_{H^s(\mt^n)} \leq C(1 + t)^{-\frac{m - s}{2(1 + r)}} \quad \text{for any } s \in [0, m].
\end{align}
\end{theorem}
\begin{theorem}[Nonlinear stability II]\label{thm}
	Let $n\geq 2$, $r>n-1$ and $\mu=1$, $\nu=0$. Let \(m \in \mathbb{N}\) satisfy
\begin{align}\label{m1}
m > 3 + 2r + \frac{n}{2}.
\end{align}
Assume that \(\tilde{b}\) satisfies the Diophantine condition, and that the initial data \((u_0, b_0) \in H^m(\mathbb{T}^n)\) satisfy
\[
\mathrm{div}\, u_0 = \mathrm{div}\, b_0 = 0, \quad
\int_{\mathbb{T}^n} u_0\, \mathrm{d}x = \int_{\mathbb{T}^n} b_0\, \mathrm{d}x = 0,
\]
and
\begin{align}\label{smallconditionq}
	\|u_0\|_{H^m(\mt^n)}+\|b_0\|_{H^m(\mt^n)}\leq \varepsilon,
\end{align}
for some sufficiently small constant \(\varepsilon > 0\).
Then the incompressible MHD equations \eqref{equation} admit a unique global classical solution \((u, b)\) satisfying
\begin{align}\label{eq:1.2}
\sup_{t \in [0, \infty)} \lVert (u, b)(t) \rVert_{H^m}+ \int_0^\infty \lVert \nabla u(t) \rVert_{H^m}^2 \,{d}t
\leq C \lVert (u_0, b_0) \rVert_{H^m},
\end{align}
and the decay estimate
\begin{align}\label{finaldecay}
\lVert (u, b)(t) \rVert_{H^s(\mt^n)} \leq C(1 + t)^{-\frac{m - s}{2(1 + r)}} \quad \text{for any } s \in [0, m].
\end{align}
\end{theorem}
\begin{remark}
We compare our regularity assumptions with those in \cite{XieJiu-CVPDE2024}, where the condition $m > 3(r+1)$ was imposed.

\smallskip
(1) Theorem \ref{thm.2} (pure magnetic diffusion) requires
$
m > 4 + 2r + \tfrac{n}{2},\, r > n-1.
$
Taking into account that $m \in \mathbb{N}$, for low dimensions $n=2,3$, our regularity threshold covers that of \cite{XieJiu-CVPDE2024} when
$r \in [\tfrac{4}{3}, \tfrac{3}{2}) \cup [\tfrac{5}{3}, +\infty)$ for $n=2$, and
$r \in (2, \tfrac{9}{4}) \cup [\tfrac{7}{3}, +\infty)$ for $n=3$;
for higher dimensions $n \ge 4$, our result completely includes that of \cite{XieJiu-CVPDE2024}$\!$ and is less restrictive in general.

\smallskip
(2) Theorem \ref{thm} (pure viscosity) further relaxes the condition to
$
m > 3 + 2r + \tfrac{n}{2},
$
under which our regularity assumption again covers that of \cite{XieJiu-CVPDE2024} for all $n \ge 2$.

\smallskip
In summary, our results not only encompass the previous regularity conditions in \cite{XieJiu-CVPDE2024},
but also yield improved Sobolev thresholds for the initial data in all dimension $n \ge 2$, with the advantage becoming more pronounced as either the spatial dimension $n$ or the index $r$ increases.
\end{remark}

\begin{remark}
Compared with \cite{XieJiu-CVPDE2024}, where the decay estimates \eqref{2024finaldecay} are valid for Sobolev indices \(r+1 \le s \le m\), our results fill the gap and extend the decay analysis to {\it all Sobolev spaces} \(H^s\) with \(0 \le s \le m\), thereby providing a complete picture of temporal decay from $L^2$ up to $H^m$.
\end{remark}

\begin{remark}
A distinctive feature of the degenerate MHD system is also revealed: the $L^2$ decay rate is enhanced for higher-order initial data, while the $H^m$ decay rate remains unchanged. This behavior is in sharp contrast with classical parabolic equations, for which higher Sobolev norms typically decay faster.
\end{remark}
\begin{remark}
Theorems \ref{thm.2} and \ref{thm} hold for all $n \ge 2$, where the dimension $n$ influences both the minimal regularity and the decay rates.
\end{remark}
\begin{remark}
The approach developed here is robust and dimension-uniform, providing decay estimates and asymptotic stability in Sobolev spaces of much lower regularity. It can be adapted to other partially dissipative PDEs on the periodic box $\mathbb{T}^n$.
\end{remark}
\begin{remark}
We believe that it is a challenging problem to drop the {\it Diophantine condition} \eqref{Diophantine} in our theorems.
\end{remark}

\subsection{On reduced regularity and decay mechanisms}\label{mechanisms}
In the following, we explain how we can {\it greatly reduce} the Sobolev regularity required for stability and establish decay rates of $(u,b)$ in $H^s$ for all $s \in [0,m]$, compared with previous works \cite{chen-2022-3dmhd-Diophant,zhai-2023-2dmhdstability-Diophant,XieJiu-CVPDE2024}.
For clarity, we focus on the magnetic diffusive case $\mu=0$ and $\nu=1$ (Theorem \ref{thm.2}); the velocity-diffusive case $\mu=1$, $\nu=0$ can be treated similarly.
Finally, we briefly discuss the origin of the different regularity thresholds in Theorems \ref{thm.2} and \ref{thm} (see \eqref{m} and \eqref{m1}) and why Theorem \ref{thm} allows for lower-order initial data.

A key step in establishing global stability of
\eqref{equation} with only magnetic diffusion in the framework of $H^m(\mathbb{T}^n)$, the most important point is how to control the quantity
\begin{equation}\label{ana1}
 \int_{0}^{\infty}\big(\|\nabla u(\tau)\|_{L^\infty(\mathbb{T}^n)}
   +\|\nabla b(\tau)\|_{L^\infty(\mathbb{T}^n)}\big)\,\dd \tau,
\end{equation}
which appears in the energy inequality \eqref{energy1'} and the relaxation of the regularity requirement on $m$ precisely originates from this control.

Previous works \cite{chen-2022-3dmhd-Diophant,zhai-2023-2dmhdstability-Diophant,XieJiu-CVPDE2024} derived decay estimates for \eqref{equation} using the classical Sobolev embedding
\begin{equation*}
 H^s(\mathbb{T}^n) \hookrightarrow L^\infty(\mathbb{T}^n), \qquad s>\tfrac n2,
\end{equation*}
which implies
\begin{align*}
\lVert(\nabla u, \nabla b)\rVert_{L^{\infty}} \leq C \lVert(u, b) \rVert_{H^{s}}, \quad \text{with } \,s > 1+ \tfrac{n}{2}.
\end{align*}
Thus, controlling \eqref{ana1} reduces to choosing a Sobolev index $s>\frac{n}{2}+1$ such that the corresponding decay rate exceeds $(1+t)^{-1}$. In particular, Xie, Jiu, and Liu \cite{XieJiu-CVPDE2024} established the temporal decay estimate
\begin{align*}
 \|u(t)\|_{H^{r+1}(\mathbb{T}^n)}+\|b(t)\|_{H^{r+1}(\mathbb{T}^n)}
 \leq C(1+t)^{-\frac{m-(r+1)}{2(r+1)}},
\end{align*}
which requires $m>3(1+r)$ to achieve decay faster than $(1+t)^{-1}$, lowering the Sobolev regularity threshold compared with earlier works \cite{chen-2022-3dmhd-Diophant,zhai-2023-2dmhdstability-Diophant}. However, the result in \cite{XieJiu-CVPDE2024} is valid only for $s\ge r+1$ and provides no decay information below this level.

Here, instead of relying on Sobolev embedding, we analyze the key quantity \eqref{ana1} in the frequency space:
\begin{align*}
\int_0^T \big(\|\nabla u(t)\|_{L^\infty} + \|\nabla b(t)\|_{L^\infty}\big)\,dt
\lesssim \sum_{k\neq 0}\int_0^T |k|\,|\hat u(t,k)|\,dt
+ \sum_{k\neq 0}\int_0^T |k|\,|\hat b(t,k)|\,dt.
\end{align*}
Exploiting the system structure, the Diophantine condition, periodicity, Fourier analysis, and a standard bootstrap argument, we obtain that for $m>4 + 2r + \frac{n}{2}$,
\begin{align*}
\int_0^T \big(\|\nabla u(t)\|_{L^\infty} + \|\nabla b(t)\|_{L^\infty}\big)\,dt
\leq C \|(u_0,b_0)\|_{H^m}, \quad \forall\, T>0,
\end{align*}
see Proposition \ref{prop6.1a} for details. Moreover, to extend decay rates of $(u,b)$ in $H^s$ for all $s \in [0, m]$, we employ a Poincar\'{e}-type inequality induced by the Diophantine condition on $\tilde{b}$ (Lemma \ref{2.1}), which ensures decay for low Sobolev norms. Combining this with the boundedness of high norms from the global existence result then yields decay across the entire range $s \in [0,m]$.
Finally, we clarify why the regularity threshold \eqref{m1} in Theorem \ref{thm} is lower than that in Theorem \ref{thm.2} (see \eqref{m}). In the velocity diffusive case $(\mu,\nu)=(1,0)$, it suffices to control only
\[
 \int_{0}^{\infty}\|\nabla u(\tau)\|_{L^\infty(\mathbb{T}^n)}\,\dd \tau,
\]
since the magnetic field $b$ can be effectively controlled through its coupling with $u$. As a result, the initial data can be taken in a lower-order Sobolev space (see Propositions \ref{prop5.1} and \ref{prop5.2}).

We emphasize that these mechanisms not only account for the reduced regularity thresholds and decay results, but also illuminate the intrinsic anisotropic and degenerate structures of the incompressible MHD system under the Diophantine condition.

\subsection{Organization of the paper}
The remainder of this paper is organized as follows. Section \ref{2} collects preliminary lemmas and tools, including Fourier multiplier estimates and properties of the Diophantine condition. Section \ref{sec:kernel} provides an integral representation of the solution via spectral analysis, with the frequency space decomposed into subdomains to obtain detailed kernel estimates. Section \ref{2/3} establishes linear stability, while Sections \ref{s:7} and \ref{s:6} study the nonlinear problem in the magnetic diffusion $(\mu,\nu)=(0,1)$ and velocity diffusion $(\mu,\nu)=(1,0)$ regimes, respectively. In both cases, uniform energy bounds are derived, yielding global existence, stability, and time-decay estimates for the solutions.

\section{Preliminaries}\label{2}
In this section we collect some analytic tools and inequalities used in the subsequent analysis. We begin with the spatial averages of the solution $(u,b)$.
\begin{lemma}\label{lem:mean}
Let $(u,b)$ be a smooth solution to \eqref{equation} that satisfies
\[
\operatorname{div} u_0 = \operatorname{div} b_0 = 0, \quad \int_{\mathbb{T}^n} u_0\, {d}x = \int_{\mathbb{T}^n} b_0\, {d}x = 0.
\]
Then, for all \( t \geq 0 \), it holds that
\begin{equation}\label{meanzero2}
\int_{\mathbb{T}^n} u(x,t)\, {d}x = \int_{\mathbb{T}^n} b(x,t)\, {d}x = \w{u}(0,t)=\w{b}(0,t)=0.
\end{equation}
\begin{proof}
We integrate the first equation in \eqref{equation} over \( \mathbb{T}^n \):
\[
\frac{d}{dt} \int_{\mathbb{T}^n} u\,dx - \mu \int_{\mathbb{T}^n} \Delta u\,dx + \int_{\mathbb{T}^n} u \cdot \nabla u\,dx + \int_{\mathbb{T}^n} \nabla p\,dx = \int_{\mathbb{T}^n} \tilde{b} \cdot \nabla b\,dx + \int_{\mathbb{T}^n} b \cdot \nabla b\,dx.
\]
By integration by parts, the divergence-free condition, and the periodic boundary conditions, we obtain
\[
\int_{\mathbb{T}^n} \Delta u \, dx =\int_{\mathbb{T}^n} u \cdot \nabla u \, dx = \int_{\mathbb{T}^n} \nabla P \, dx = \int_{\mathbb{T}^n} \tilde{b} \cdot \nabla b \, dx = \int_{\mathbb{T}^n} b \cdot \nabla b \, dx = 0.
\]
Hence,
\[
\frac{d}{dt} \int_{\mathbb{T}^n} u\,dx = 0.
\]
It follows that
\[
\int_{\mathbb{T}^n} u(x,t)\,dx = \int_{\mathbb{T}^n} u_0(x)\,dx=0.
\]
Similarly, integrating the second equation in \eqref{equation} and using the same arguments, we obtain
\[
\int_{\mathbb{T}^n} b(x,t)\,dx = \int_{\mathbb{T}^n} b_0(x)\,dx = 0.
\]
This completes the proof.
\end{proof}
\end{lemma}

With the mean-zero condition in hand, we now state the following Poincar\'{e}-type inequality under a Diophantine condition, which plays a key role in our setting with directional transport effects.
\begin{lemma}{\rm (\cite{chen-2022-3dmhd-Diophant,Jiu-Liu-Xie-compressibleMHD})}\label{2.1}
Let $\tb\in\mr^n$ be a given vector satisfying the Diophantine condition \eqref{Diophantine}. For any $s\in\mr$, there exists a constant $c$ such that if $f\in H^{s+r+1(\mathbb{T}^n)}$ satisfies $\int_{\mathbb{T}^n}f\dd x=0$, then
	\begin{align}\label{D1}
		\|f\|_{H^s(\mathbb{T}^n)}\leq c \|\tilde{b}\cdot\nabla f\|_{H^{s+r}(\mathbb{T}^n)}.
	\end{align}
\end{lemma}
\begin{proof}
According to Plancherel's theorem and \eqref{Diophantine}, it holds that
\begin{align*}
	\|\tilde{b}\cdot\nabla f\|_{H^{s+r}}^2&=\sum_{k\in\mz^n}(1+|k|^2)^{s+r}|\tilde{b}\cdot k|^2|\hat{f}(k)|^2\\
	&=\sum_{k\in\mz^n\backslash\{0\}}(1+|k|^2)^{s+r}|\tilde{b}\cdot k|^2|\hat{f}(k)|^2\\
	&\geq c\sum_{k\in\mz^n\backslash\{0\}}(1+|k|^2)^{s+r}|k|^{-2r}|\hat{f}(k)|^2\\
	&\geq c\sum_{k\in\mz^n\backslash\{0\}}(1+|k|^2)^{s}|\hat{f}(k)|^2\\
	&=c\|f\|^2_{H^s},
\end{align*}
which completes the proof.
\end{proof}

The following lemma records a useful property of fractional operators acting on zero-mean functions.
\begin{lemma}\label{lem:zero-mean-fractional}
Let \( u \in \mathcal{S}'(\mathbb{T}^n) \) be a tempered distribution with zero mean, i.e.
\[
\int_{\mathbb{T}^n} u(x)\,\mathrm{d}x = 0.
\]
Assume the Fourier convention on the torus $\mathbb{T}^n=[0,2\pi]^n$ given by
\[
\hat u(k)=\frac{1}{(2\pi)^n}\int_{\mathbb{T}^n} u(x)e^{-i k\cdot x}\,\mathrm{d}x,\qquad
u(x)=\sum_{k\in\mathbb Z^n}\hat u(k)e^{i k\cdot x}.
\]
If the fractional Laplacian \(\Lambda^s\) is defined in the Fourier sense by
\[
\widehat{\Lambda^s u}(k)=|k|^{s}\,\hat u(k),
\]
then for any \( s\in\mathbb R\), we have
\[
\int_{\mathbb{T}^n} \Lambda^s u(x)\,\mathrm{d}x = 0.
\]
\end{lemma}
\begin{proof}
The zero-mean condition implies \(\hat u(0)=0\), so
\[
u(x)=\sum_{k\in\mathbb Z^n\setminus\{0\}}\hat u(k)e^{i k\cdot x}.
\]
Applying the definition of \(\Lambda^s\) yields
\[
\Lambda^s u(x)=\sum_{k\in\mathbb Z^n\setminus\{0\}}|k|^s\hat u(k)e^{i k\cdot x}.
\]
Integrating term by term gives
\[
\int_{\mathbb{T}^n}\Lambda^s u(x)\,\mathrm{d}x
   =\sum_{k\in\mathbb Z^n\setminus\{0\}}|k|^s\hat u(k)
     \int_{\mathbb{T}^n} e^{i k\cdot x}\,\mathrm{d}x.
\]
Since \(\int_{\mathbb{T}^n} e^{i k\cdot x}\,\mathrm{d}x=0\) for every \(k\ne0\), each term in the above sum vanishes, and hence
\[
\int_{\mathbb{T}^n}\Lambda^s u(x)\,\mathrm{d}x=0.
\]
This completes the proof.
\end{proof}

Finally, we recall a commutator estimate and a calculus
inequality, which will be used to control nonlinear terms in the energy estimates and can be found in \cite{kato-1988-CommutatorEstimatesEuler, kenig}.
\begin{lemma}\label{jiaohuanzi}
Let \( s > 0 \), \( 1 \leq p, p_1, p_2, q_1, q_2 \leq \infty \) and \( \frac{1}{p} = \frac{1}{p_1} + \frac{1}{q_1} = \frac{1}{p_2} + \frac{1}{q_2} \). Then, there exists an absolutely positive constant \( C \) such that
\begin{itemize}
    \item for any \( f \in W^{1,p_1} \cap W^{s,q_2} \) and \( g \in L^{p_2} \cap W^{s-1,q_1} \),
    \[
    \| \Lambda^s (fg) - f \Lambda^s g \|_{L^p} \leq C \big( \| \nabla f \|_{L^{p_1}} \| \Lambda^{s-1} g \|_{L^{q_1}} + \| g \|_{L^{p_2}} \| \Lambda^s f \|_{L^{q_2}} \big);
    \]
    \item for \( f \in L^{p_1} \cap W^{s,q_2} \) and \( g \in L^{p_2} \cap W^{s,q_1} \),
    \[
    \| \Lambda^s (fg) \|_{L^p} \leq C \big( \| f \|_{L^{p_1}} \| \Lambda^s g \|_{L^{q_1}} + \| g \|_{L^{p_2}} \| \Lambda^s f \|_{L^{q_2}} \big).
    \]
\end{itemize}
\end{lemma}

\section{Integral representation and kernel estimates} \label{sec:kernel}

In this section, we derive an explicit integral representation for the solution of system \eqref{equation} and establish upper bounds for the associated kernel functions using spectral analysis. For clarity, we present the detailed derivation for the case $\mu=1$ and $\nu=0$, where viscosity is present but magnetic diffusion is absent; the complementary case $\mu=0$ and $\nu=1$ can be treated similarly. Throughout this section, we use \(\langle \cdot, \cdot \rangle\) to denote the standard inner product on \(\mathbb{C}^n\) for any \(n \geq 2\).

In contrast to \cite{XieJiu-CVPDE2024}, where the analysis relies on a hidden wave structure and produces identical integral representations for $u$ and $b$ in the two degenerate cases, our spectral approach yields distinct representations for $u$ and $b$. These structural differences directly affect the dissipation mechanisms acting on each component, resulting in different decay behaviors for $u$ and $b$ in the two settings.

We first establish the integral representation of the solution in Proposition \ref{pro1}, and then derive upper bounds for the kernel functions in Proposition \ref{pro2}, which are essential for quantifying the linear decay rates of $u$ and $b$ in the subsequent analysis.
\begin{prop}\label{pro1}
Let $(u,b)$ be a solution to system \eqref{equation} with $\mu=1$ and $\nu=0$. Then $(u,b)$ can be estimated as follows
\begin{align}
\left\{\begin{aligned}\label{2.4}
&|\hat{u}|\leq\left(|\widehat{K_2}|+
|\widehat{K_3}|\right)|\hat{\boldsymbol{\psi}}_0|
 +\int_0^t\left(|\widehat{K_2}(t-\tau)||\widehat{N}(\tau)|
+|\widehat{K_3}(t-\tau)||\widehat{N_1}(\tau)|\right)d\tau,\\[1ex]
&|\hat{b}|\leq\left(|\widehat{K_1}|+
|\widehat{K_3}|\right)|\hat{\boldsymbol{\psi}}_0| +\int_0^t\left(|\widehat{K_1}(t-\tau)||\widehat{N}(\tau)|
+|\widehat{K_3}(t-\tau)||\widehat{N_2}(\tau)|\right)d\tau,
\end{aligned}\right.
\end{align}
where $\boldsymbol{\psi}_0:=(u_0, b_0)^T$, $N=({N_1}, {N_2})^T$, ${N_1}:=\mathbb{P}(u \cdot \nabla u-b \cdot \nabla b)$, ${N_2}:=u \cdot \nabla b- b \cdot \nabla u
$ and $\mathbb{P}$ denotes the Helmholtz-Leray projection operator. The kernel functions $\widehat{K_1},\, \widehat{K_2}$ and $\widehat{K_3}$ are given by
\begin{align}\label{2.5'}
&|\widehat{K_1}|
=\frac{\left| e^{-\lambda_{-}t} - e^{-\lambda_{+}t} \right|}{|\lambda_{+} - \lambda_{-}|} \sqrt{ |\lambda_{+}|^2 + |\tilde{b} \cdot k|^2 }=:|\widehat{G}|\sqrt{ |\lambda_{+}|^2 + |\tilde{b} \cdot k|^2 },\notag\\[1ex]
&|\widehat{K_2}|
=\frac{\left| e^{-\lambda_{-}t} - e^{-\lambda_{+}t} \right|}{|\lambda_{+} - \lambda_{-}|} \sqrt{ |\lambda_{+}|^2 + |\tilde{b} \cdot k|^2 }\left| \frac{\tilde{b}\cdot k}{\lambda_{+}} \right|=:|\widehat{G}|\sqrt{ |\lambda_{+}|^2 + |\tilde{b} \cdot k|^2 }\left| \frac{\tilde{b}\cdot k}{\lambda_{+}} \right|, \\[1ex]
&\widehat{K_3} =e^{-\lambda_{+}t},\notag
\end{align}
where $\widehat{G}$ is Fourier multiplier operator
\begin{align}\label{2.5''}
\widehat{G} := \frac{e^{-\lambda_{-}t} - e^{-\lambda_{+}t}}{\lambda_{+} - \lambda_{-}}
\end{align}
with $\lambda_+$ and $\lambda_-$ being the roots of the characteristic equation,
\begin{align}\label{2.7}
\lambda^2 - |k|^2\lambda + |\tilde{b}\cdot k|^2=0
\end{align}
or
\begin{align*}
\lambda_\pm(k) = \frac{|k|^2 \pm \sqrt{|k|^4 - 4|\tilde{b}\cdot k|^2}}{2}.
\end{align*}
When $\lambda_+ = \lambda_-$, \eqref{2.4} remains valid if we replace $\widehat{G}$ in \eqref{2.5''} by their corresponding limit form, namely
\begin{equation*}
\widehat{G}=\lim_{\lambda_- \to \lambda_+} \frac{e^{-\lambda_{-}t} - e^{-\lambda_{+}t}}{\lambda_{+} - \lambda_{-}}=te^{-\lambda_{+}t}.
\end{equation*}
\end{prop}
The proof follows from spectral analysis and will be given below. The complementary case $\mu=0$ and $\nu=1$ is stated in the following Proposition.
\begin{prop}\label{cor:mu0nu1}
For system \eqref{equation} with $\mu=0$ and $\nu=1$, the solution $(u,b)$ satisfies
\begin{align}
\left\{
\begin{aligned}\label{2.4s}
&|\hat{u}(t)| \le \big(|\widehat{K_1}| + |\widehat{K_3}|\big) |\hat{\boldsymbol{\psi}}_0|
   + \int_0^t \Big( |\widehat{K_1}(t-\tau)|\,|\widehat{N}(\tau)|
   + |\widehat{K_3}(t-\tau)|\,|\widehat{N_1}(\tau)| \Big) \, d\tau,\\[1ex]
&|\hat{b}(t)| \le \big(|\widehat{K_2}| + |\widehat{K_3}|\big) |\hat{\boldsymbol{\psi}}_0|
   + \int_0^t \Big( |\widehat{K_2}(t-\tau)|\,|\widehat{N}(\tau)|
   + |\widehat{K_3}(t-\tau)|\,|\widehat{N_2}(\tau)| \Big) \, d\tau.
\end{aligned}
\right.
\end{align}
\end{prop}
The proof is identical to that of Proposition \ref{pro1} and is omitted.
\begin{proof}[ Proof of Proposition {\rm \ref{pro1}} ]
 We start by separating the linear terms in $\eqref{equation}$ from the nonlinear ones. Applying the Helmholtz-Leray projection operator
$\mathbb{P}: = \text{Id} - \nabla \Delta^{-1} \nabla \cdot$
to the velocity equation in $\eqref{equation}_1$. Then, \eqref{equation} is converted into
\begin{equation}\label{1.41}
\left\{ \begin{array}{l}
{\partial _t}u- \Delta u -\tilde{b} \cdot \nabla b+ N_1(u, b)=0,\\[1ex]
{\partial _t}b-\tilde{b} \cdot \nabla u+N_2(u, b) =0,\\[1ex]
\mathrm{div}\,u = \mathrm{div}\,b = 0,\\[1ex]
(u,b)|_{t=0} = (u_0, b_0),
\end{array} \right.
\end{equation}
where $N_1(u, b)$ and $N_2(u, b)$ are the nonlinear terms
\begin{equation}\label{1.411}
N_1(u, b)=\mathbb{P}(u \cdot \nabla u-b \cdot \nabla b),\,\,\,\,\,\,\,
N_2(u, b)=u \cdot \nabla b-  b \cdot \nabla u.
\end{equation}
Taking the Fourier transform to \eqref{1.41}, we find that, for \(k \in \mathbb{Z}^n \setminus \{0\}\),
\begin{equation}\label{97}
\left\{ \begin{array}{l}
\partial_t \hat{u} + |k|^2 \hat{u} - i (\tilde{b} \cdot k) \hat{b}  + \left( I - \frac{k \otimes k}{|k|^2} \right) \left( \widehat{u \cdot \nabla u} - \widehat{b \cdot \nabla b} \right)=0,\\[1ex]
\partial_t \widehat{b} -i(\widetilde{b}\cdot k)\widehat{u}+ \widehat{u\cdot\nabla b} - \widehat{b\cdot\nabla u} =0.
\end{array} \right.
\end{equation}
To better capture the dissipative mechanism and coupling effects, we rewrite the system in vector form \(\hat{\boldsymbol{\psi}} := (\hat{u}, \hat{b})^T\), which yields
\begin{align}\label{b}
\partial_t \hat{\boldsymbol{\psi}} + \emph{M} \hat{\boldsymbol{\psi}} + \widehat{N}(u, b) = 0,
\end{align}
where
\begin{align}\label{1.412}
\emph{M}: =
\begin{pmatrix}
|k|^2 & -i\tilde{b}\cdot k \\
-i\tilde{b}\cdot k & 0
\end{pmatrix}, \quad
\widehat{N}(u, b) :=
\begin{pmatrix}
\widehat{N_1} \\
\widehat{N_2}
\end{pmatrix}.
\end{align}
By Duhamel's principle, the solution to \eqref{b} is
\begin{align}\label{cd}
\hat{\boldsymbol{\psi}}&=e^{-\emph{M}t}\hat{\boldsymbol{\psi}}_0
-\int_0^te^{-\emph{M}(t-\tau)}\widehat{N}(u, b)(\tau)\, d\tau.
\end{align}
The fundamental solution matrix $e^{-\emph{M}t}$ can be made more explicit via the eigenvalues and eigenvectors of $M$.  Since the characteristic polynomial of the matrix \( M \) is given by
\[
\det(M - \lambda I) = \lambda^2 - |k|^2\lambda + |\tilde{b}\cdot k|^2,
\]
it admits two eigenvalues \( \lambda_\pm(k) \) with corresponding eigenvectors \( \mathbf{a}_\pm(k) \), defined by
\[
\lambda_\pm(k) = \frac{|k|^2 \pm \sqrt{|k|^4 - 4|\tilde{b}\cdot k|^2}}{2}, \quad
\mathbf{a}_\pm(k) = \begin{pmatrix} \lambda_\pm \\ -i\tilde{b}\cdot k \end{pmatrix},
\]
which satisfy the eigenvalue relation \( M\mathbf{a}_\pm(k) = \lambda_\pm(k)\mathbf{a}_\pm(k) \).

Let
\[
A_1 := (\mathbf{a}_+ \ \mathbf{a}_-), \quad
A_2 := A_1^{-1} = \frac{1}{\lambda_+ - \lambda_-} \begin{pmatrix}
1 & \frac{\lambda_-}{i\tilde{b}\cdot k} \\
-1 & -\frac{\lambda_+}{i\tilde{b}\cdot k}
\end{pmatrix} = \begin{pmatrix} \overline{\mathbf{b}}_+^{T} \\ \overline{\mathbf{b}}_-^{T} \end{pmatrix}.
\]
Thus, the matrix \( M \) is diagonalizable with
\[
M = A_1 \begin{pmatrix} \lambda_+ & 0 \\ 0 & \lambda_- \end{pmatrix} A_2 = \lambda_+ \mathbf{a}_+ \overline{\mathbf{b}}_+^{T} + \lambda_- \mathbf{a}_- \overline{\mathbf{b}}_-^{T}.
\]
Consequently, the matrix exponential \( e^{-Mt} \) is given by
\[
e^{-M t} = e^{-\lambda_+ t} \mathbf{a}_+ \overline{\mathbf{b}}_+^{T} + e^{-\lambda_- t} \mathbf{a}_- \overline{\mathbf{b}}_-^{T}.
\]
Then, by  \eqref{cd}, we get
\begin{align}\label{c}
\hat{\boldsymbol{\psi}}&=e^{-\emph{M}t}\hat{\boldsymbol{\psi}}_0
-\int_0^te^{-\emph{M}(t-\tau)}\widehat{N}\, d\tau\notag\\[1ex]
&=e^{-\lambda_+ t} \mathbf{a}_+ \overline{\mathbf{b}}_+^{T}\hat{\boldsymbol{\psi}}_0 + e^{-\lambda_- t} \mathbf{a}_- \overline{\mathbf{b}}_-^{T}\hat{\boldsymbol{\psi}}_0
-\int_0^t\left( e^{-\lambda_+ (t-\tau)} \mathbf{a}_+ \overline{\mathbf{b}}_+^{T} + e^{-\lambda_- (t-\tau)} \mathbf{a}_- \overline{\mathbf{b}}_-^{T}\right)\widehat{N}(\tau)\, d\tau\notag\\[1ex]
&=e^{-\lambda_+ t} \langle \hat{\boldsymbol{\psi}}_0, \mathbf{b}_+ \rangle \mathbf{a}_+ + e^{-\lambda_- t} \langle \hat{\boldsymbol{\psi}}_0, \mathbf{b}_- \rangle \mathbf{a}_-\notag\\[1ex]
&\quad-\int_0^t e^{-\lambda_+ (t-\tau)}\langle \widehat{N}(\tau),\mathbf{b}_{+}\rangle\mathbf{a}_{+}\, d\tau-\int_0^t e^{-\lambda_- (t-\tau)}\langle \widehat{N}(\tau),\mathbf{b}_{-}\rangle\mathbf{a}_{-}\, d\tau,
\end{align}
where we have used the identity \( \mathbf{a} \overline{\mathbf{b}}^{T} \mathbf{c} = \mathbf{a}\langle \mathbf{c}, \mathbf{b} \rangle =\langle \mathbf{c}, \mathbf{b} \rangle \mathbf{a} \)  for any column vector \( \mathbf{a} \). However, direct application of formula \eqref{c} poses challenges due to the unbounded nature of $\vert\mathbf{b}_{\pm}\vert$ in the neighborhood of the set \(\{|k|^{4}=4|\tilde{b}\cdot k|^{2}\}\). To address this concern, we resort to the following identity
\begin{align*}
\hat{\boldsymbol{\psi}}=&(e^{-\lambda_{-}t}-e^{-\lambda_{+}t})
\langle\hat{\boldsymbol{\psi}}_0,\mathbf{b}_{-}\rangle\mathbf{a}_{-}
+e^{-\lambda_{+}t}\hat{\boldsymbol{\psi}}_0\\[1ex]
&-\int_0^t(e^{-\lambda_{-}(t-\tau)}-e^{-\lambda_{+}(t-\tau)})
\langle \widehat{N}(\tau),\mathbf{b}_{-}\rangle\mathbf{a}_{-}\, d\tau
-\int_0^te^{-\lambda_{+}(t-\tau)}\widehat{N}(\tau)\, d\tau,
\end{align*}
which effectively eliminates the singularity of \( |\mathbf{b}_\pm| \). Thus, the components of \( \hat{\boldsymbol{\psi}} \) can now be expressed as follows. For \( \hat{b} \), we have
\begin{align}\label{3.1}
\hat{b}= &\langle \hat{\boldsymbol{\psi}}, e_2 \rangle\notag\\[1ex] =&(e^{-\lambda_{-}t}-e^{-\lambda_{+}t})
\langle\hat{\boldsymbol{\psi}}_0,\mathbf{b}_{-}\rangle\langle\mathbf{a}_{-}, e_2\rangle
+e^{-\lambda_{+}t}\langle\hat{\boldsymbol{\psi}}_0,e_2\rangle\notag\\[1ex]
&-\int_0^t(e^{-\lambda_{-}(t-\tau)}-e^{-\lambda_{+}(t-\tau)})
\langle \widehat{N}(\tau),\mathbf{b}_{-}\rangle\langle\mathbf{a}_{-}, e_2\rangle
 \, d\tau-\int_0^te^{-\lambda_{+}(t-\tau)}\langle \widehat{N}(\tau), e_2\rangle \, d\tau.
\end{align}
Similarly, for \( \hat{u} \), we observe that
\[
\langle \mathbf{a}_-, e_1 \rangle = \frac{\lambda_-}{-i\tilde{b}\cdot k} \langle \mathbf{a}_-, e_2 \rangle = \frac{i\tilde{b}\cdot k}{\lambda_+} \langle \mathbf{a}_-, e_2 \rangle,
\]
which yields
\begin{align}\label{3.2}
\hat{u}= &\langle \hat{\boldsymbol{\psi}}, e_1 \rangle\notag\\[1ex] =& \frac{i\tilde{b}\cdot k}{\lambda_+} (e^{-\lambda_- t} - e^{-\lambda_+ t}) \langle \hat{\boldsymbol{\psi}}_0, \mathbf{b}_- \rangle \langle \mathbf{a}_-, e_2 \rangle + e^{-\lambda_+ t} \langle \hat{\boldsymbol{\psi}}_0, e_1 \rangle\notag\\[1ex]
&-\int_0^t\frac{i\tilde{b}\cdot k}{\lambda_+}(e^{-\lambda_{-}(t-\tau)}-e^{-\lambda_{+}(t-\tau)})
\langle\widehat{N}(\tau),\mathbf{b}_{-}\rangle\langle\mathbf{a}_{-}, e_2\rangle \, d\tau-\int_0^te^{-\lambda_{+}(t-\tau)}\langle \widehat{N}(\tau), e_1\rangle\, d\tau.
\end{align}
With the representations \eqref{3.1} and \eqref{3.2} in hand, we derive
\begin{align*}
|\hat{b}|\leq& \left\lvert\left( e^{-\lambda_{-}t}-e^{-\lambda_{+}t}\right) \langle\hat{\boldsymbol{\psi}}_0,\mathbf{b}_{-}\rangle\langle\mathbf{a}_{-}, e_2\rangle\right\rvert+ \left\lvert e^{-\lambda_{+}t}\hat{b}_0 \right\rvert\notag\\[1ex]
&+\int_0^t\left\lvert\left(e^{-\lambda_{-}(t-\tau)}-e^{-\lambda_{+}(t-\tau)}\right)
\langle \widehat{N}(\tau),\mathbf{b}_{-}\rangle\langle\mathbf{a}_{-}, e_2\rangle\right\lvert  \,d\tau+\int_0^t\left\lvert e^{-\lambda_{+}(t-\tau)}\widehat{N_2}(\tau) \right\lvert \,d\tau,
\end{align*}
and
\begin{align*}
|\hat{u}|\leq& \left\lvert\frac{\tilde{b}\cdot k}{\lambda_{+}}\left( e^{-\lambda_{-}t}-e^{-\lambda_{+}t}\right) \langle\hat{\boldsymbol{\psi}}_0,\mathbf{b}_{-}\rangle\langle\mathbf{a}_{-}, e_2\rangle\right\rvert+ \left\lvert e^{-\lambda_{+}t}\hat{u}_0 \right\rvert\notag\\[1ex]
&+\int_0^t\left\lvert\frac{\tilde{b}\cdot k}{\lambda_{+}}\left(e^{-\lambda_{-}(t-\tau)}-e^{-\lambda_{+}(t-\tau)}\right)
\langle \widehat{N}(\tau),\mathbf{b}_{-}\rangle\langle\mathbf{a}_{-}, e_2\rangle\right\lvert \, d\tau+\int_0^t\left\lvert e^{-\lambda_{+}(t-\tau)}\widehat{N_1}(\tau)\right\lvert\, d\tau.
\end{align*}
From the definitions of $\mathbf{b}_{-}$ and $\mathbf{a}_{-}$,  we compute for any $\mathbf{f}\in \mathbb{C}^2$, there are
\begin{align*}
\left| (e^{-\lambda_{-}t} - e^{-\lambda_{+}t}) \langle \mathbf{f}, \mathbf{b}_{-} \rangle \langle \mathbf{a}_{-}, e_2 \rangle \right|\leq &\left| e^{-\lambda_{-}t} - e^{-\lambda_{+}t} \right| |\mathbf{b}_{-}| |\langle \mathbf{a}_{-}, e_2 \rangle| |\mathbf{f}|\notag\\[1ex]
=&{\left| e^{-\lambda_{-}t} - e^{-\lambda_{+}t} \right|}\frac{\sqrt{ |\lambda_{+}|^2 + |\tilde{b} \cdot k|^2 }}{|\lambda_{+} - \lambda_{-}|} |\mathbf{f}|\notag\\[1ex]
=&:|\widehat{K_1}(k,t)||\mathbf{f}|,
\end{align*}
and
\begin{align*}
\left| \frac{\tilde{b}\cdot k}{\lambda_{+}}(e^{-\lambda_{-}t} - e^{-\lambda_{+}t}) \langle \mathbf{f}, \mathbf{b}_{-} \rangle \langle \mathbf{a}_{-}, e_2 \rangle \right|
\leq& \left| e^{-\lambda_{-}t} - e^{-\lambda_{+}t} \right| |\mathbf{b}_{-}| |\langle \mathbf{a}_{-}, e_2 \rangle|\left| \frac{\tilde{b}\cdot k}{\lambda_{+}} \right| |\mathbf{f}|\notag\\[1ex]
=&\frac{\left| e^{-\lambda_{-}t} - e^{-\lambda_{+}t} \right|}{|\lambda_{+} - \lambda_{-}|} {\sqrt{ |\lambda_{+}|^2 + |\tilde{b} \cdot k|^2 }}\left| \frac{\tilde{b}\cdot k}{\lambda_{+}} \right||\mathbf{f}|\notag\\[1ex]
=&:|\widehat{K_2}(k,t)||\mathbf{f}|.
\end{align*}
To simplify the notation, we define
$$
\widehat{K_3}(k,t): = e^{-\lambda_{+}t},\quad  \widehat{G}(k,t) := \frac{e^{-\lambda_{-}t} - e^{-\lambda_{+}t}}{\lambda_{+} - \lambda_{-}},
$$
and then
$$
|\widehat{K_1}(k,t)| = |\widehat{G}(k,t)|{\sqrt{ |\lambda_{+}|^2 + |\tilde{b} \cdot k|^2 }},\quad |\widehat{K_2}(k,t)| = |\widehat{G}(k,t)|{\sqrt{ |\lambda_{+}|^2 + |\tilde{b} \cdot k|^2 }}\left| \frac{\tilde{b}\cdot k}{\lambda_{+}} \right| .
$$
Thus, we get
\begin{align*}
|\hat{b}|\leq|\widehat{K_1}||\hat{\boldsymbol{\psi}}_0|+
|\widehat{K_3}||\hat{b}_0| +\int_0^t|\widehat{K_1}(t-\tau)||\widehat{N}(\tau)|\, d\tau
+\int_0^t|\widehat{K_3}(t-\tau)||\widehat{N_2}(\tau)|\, d\tau,
\end{align*}
and
\begin{align*}
|\hat{u}|\leq|\widehat{K_2}||\hat{\boldsymbol{\psi}}_0|+
|\widehat{K_3}||\hat{u}_0| +\int_0^t|\widehat{K_2}(t-\tau)||\widehat{N}(\tau)|\, d\tau
+\int_0^t|\widehat{K_3}(t-\tau)||\widehat{N_1}(\tau)|\,d\tau.
\end{align*}
This completes the proof of Proposition \ref{pro1}.
\end{proof}

We now analyze the behavior of the kernel functions $\widehat{K_1}$--$\widehat{K_3}$, which are anisotropic and inhomogeneous Fourier multipliers. To accurately capture their decay properties, it is natural to decompose the frequency space $\mathbb{Z}^n \setminus \{0\}$ into suitable subregions. The following proposition establishes uniform upper bounds for $\widehat{K_1}$--$\widehat{K_3}$ in each of these regions.
\begin{prop}\label{pro2}
Partition the frequency space $\mathbb{Z}^n\setminus\{0\}$ into the following three subdomains:
\begin{align*}
&S_1 := \left\{ k \in \mathbb{Z}^n \setminus \{0\} : |k|^4 - 4|\tilde{b}\cdot k|^2 \leq 0 \right\},\notag\\[1ex]
&S_2 := \left\{ k \in \mathbb{Z}^n \setminus \{0\} : 0 < |k|^4 - 4|\tilde{b}\cdot k|^2 \leq \tfrac{1}{4}|k|^4 \right\},\notag\\[1ex]
&S_3 := \left\{ k \in \mathbb{Z}^n \setminus \{0\} : |k|^4 - 4|\tilde{b}\cdot k|^2 > \tfrac{1}{4}|k|^4 \right\},
\end{align*}
where $S_1$, $S_2$ and $S_3$ represent {\it low frequency space}, {\it medium frequency space} and {\it high frequency space} respectively. Then, there exists an absolute constant $C > 0$ such that,
\newline
(1)\,\,for any $k\in S_1\cup S_2 $,
\begin{align}
|\widehat{K_1}|, |\widehat{K_2}|\leq C e^{-\frac{|k|^2}{8}t};\notag
\end{align}
\newline
(2)\,\, for any $k\in S_{3} $,
\begin{align}
|\widehat{K_1}|\leq Ce^{-\frac{(\tilde{b}\cdot k)^2}{|k|^2}t},\quad
|\widehat{K_2}|\leq C \frac{|\tilde{b}\cdot k|}{|k|^2} e^{-\frac{(\tilde{b}\cdot k)^2}{|k|^2}t};\notag
\end{align}
\newline
(3)\,\,for any $ k \in \mathbb{Z}^n \setminus \{0\}$,
\begin{align}
 |\widehat{K_3}|\leq e^{-\frac{|k|^2}{2}t}.\notag
\end{align}
\end{prop}

\begin{proof}[Proof of Proposition {\rm \ref{pro2}}]
To estimate \( |\widehat{K_1}| \), \( |\widehat{K_2}| \), and \( |\widehat{K_3}| \), we use their representations in terms of the characteristic roots \( \lambda_+, \lambda_- \) and the Fourier multiplier \( \widehat{G}(k) \), as given in \eqref{2.5'}, \eqref{2.7}, and \eqref{2.5''}. We proceed in steps.

\medskip
\noindent
\textbf{Step 1: Estimate of \( |\widehat{K_3}| \).} By definition \eqref{2.5'}, \( \widehat{K_3} = e^{-\lambda_+ t} \). Since \( \Re \lambda_+ \geq \frac{1}{2}|k|^2 \) for all \( k \neq 0 \), we immediately obtain
\[
|\widehat{K_3}| \leq e^{-\frac{|k|^2}{2}t}, \quad \text{for all } k \in \mathbb{Z}^n \setminus \{0\}.
\]

\medskip
\noindent
\textbf{Step 2: Estimate for $\widehat{G}(k,t)$.} Recall that
\begin{align*}
\widehat{G}(k,t) := \frac{e^{-\lambda_{-}t} - e^{-\lambda_{+}t}}{\lambda_{+} - \lambda_{-}}.
\end{align*}
We now estimate $\widehat{G}(k,t)$ for each frequency region $S_1$, $S_2$, and $S_3$.

\medskip
\noindent
\underline{\textbf{Case 1: $k \in S_1$}}. In this case, the characteristic roots $\lambda_\pm$ form a pair of complex conjugates. Let
\[
\sigma := \sqrt{4|\tilde{b}\cdot k|^2 - |k|^4},
\]
so that $\lambda_\pm = \frac{|k|^2}{2} \pm i \frac{\sigma}{2}$. Then, using Euler's formula,
\[
\widehat{G}(k,t) = e^{-\frac{|k|^2}{2}t} \cdot \frac{e^{i\frac{\sigma}{2}t} - e^{-i\frac{\sigma}{2}t}}{i\sigma}
= t e^{-\frac{|k|^2}{2}t} \cdot \frac{\sin\left( \frac{\sigma}{2}t \right)}{\frac{\sigma}{2}t}.
\]
Since $\left|\frac{\sin \theta}{\theta}\right| \leq 1$ for all $\theta \in \mathbb{R}$, we have
\[
|\widehat{G}(k,t)| \leq t e^{-\frac{|k|^2}{2}t}.
\]
Applying the standard inequality
\begin{align}\label{ye}
y^n e^{-y} \leq C_n, \quad \text{for all } y>0,\; n \in \mathbb{N},
\end{align}
with $y = \frac{|k|^2}{4}t$, $n=1$, yields
\[
|\widehat{G}(k,t)| \leq C e^{-\frac{1}{4}|k|^2 t}.
\]

\medskip
\noindent
\underline{\textbf{Case 2: $k \in S_2$}}. Here the roots $\lambda_\pm$ are real and distinct. Noting that
\begin{align}
&-\frac{3|k|^2}{4}\leq-\lambda_+=\frac{-|k|^2-\sqrt{|k|^{4}-4|\tilde{b}\cdot k|^{2}}}{2}<-\frac{|k|^2}{2},\label{S_2}\\
&-\frac{|k|^2}{2}<-\lambda_- =\frac{-|k|^2+\sqrt{|k|^{4}-4|\tilde{b}\cdot k|^{2}}}{2}\leq-\frac{|k|^2}{4},\notag
\end{align}
the mean-value theorem implies that for some $\zeta \in (-\lambda_+,-\lambda_-)$,
\[
|\widehat{G}(k,t)| = \left| \frac{e^{-\lambda_- t} - e^{-\lambda_+ t}}{\lambda_+ - \lambda_-} \right| = t e^{\zeta t}.
\]
Since $\zeta \leq -\frac{1}{4}|k|^2$, it follows that
\[
|\widehat{G}(k,t)| \leq t e^{-\frac{1}{4}|k|^2 t} \leq C \frac{1}{|k|^2} e^{-\frac{1}{8}|k|^2 t},
\]
where we used again \eqref{ye} with $y = \frac{1}{8}|k|^2 t$.

\medskip
\noindent
\underline{\textbf{Case 3: $k \in S_3$}}. The root $-\lambda_+$ satisfies
\begin{align}\label{S_3}
-|k|^2 \leq -\lambda_+ = -\frac{|k|^2}{2} \left( 1 + \sqrt{1 - 4\frac{(\tilde{b} \cdot k)^2}{|k|^4}} \right) < -\frac{3}{4}|k|^2.
\end{align}
In contrast, the root $-\lambda_-$ may approach zero. To handle this, we rewrite it as
\[
-\lambda_- = -\frac{|k|^2}{2} \left(1 - \sqrt{1 - 4\frac{(\tilde{b} \cdot k)^2}{|k|^4}} \right)
= \frac{-2 \frac{(\tilde{b} \cdot k)^2}{|k|^2}}{1 + \sqrt{1 - 4\frac{(\tilde{b} \cdot k)^2}{|k|^4}}}.
\]
Since $k \in S_3$ implies $\frac{1}{4} <1 - 4\frac{(\tilde{b}\cdot k)^2}{|k|^4} \leq 1$, we have
\[
-\frac{4(\tilde{b}\cdot k)^2}{3|k|^2} < -\lambda_- \leq -\frac{(\tilde{b}\cdot k)^2}{|k|^2}, \quad
\lambda_+ - \lambda_- = \sqrt{|k|^4 - 4|\tilde{b}\cdot k|^2} \geq \frac{1}{2}|k|^2.
\]
Therefore,
\[
|\widehat{G}(k,t)| \leq C \frac{e^{-\frac{3}{4}|k|^2 t} + e^{-\frac{(\tilde{b}\cdot k)^2}{|k|^2}t}}{|k|^2}
\leq C \frac{e^{-\frac{(\tilde{b}\cdot k)^2}{|k|^2}t}}{|k|^2}.
\]

\medskip
\noindent
\textbf{Step 3: Estimates for \( |\widehat{K_1}| \) and \( |\widehat{K_2}| \).}
We first recall from \eqref{2.5'} that
\begin{gather}
\begin{split}\label{K_12}
|\widehat{K_1}|^2 &= |\widehat{G}|^2 \left( |\lambda_+|^2 + |\tilde{b} \cdot k|^2 \right), \\
|\widehat{K_2}|^2 &= |\widehat{G}|^2 \left( |\lambda_+|^2 + |\tilde{b} \cdot k|^2 \right) \left| \frac{\tilde{b} \cdot k}{\lambda_+} \right|^2.
\end{split}
\end{gather}
Thus, we need to estimate the quantities \( |\lambda_+|^2 + |\tilde{b} \cdot k|^2 \) and \( \left| \frac{\tilde{b} \cdot k}{\lambda_+} \right|^2 \) in each region.

\medskip
\noindent
\underline{\textbf{Estimate of \( |\lambda_+|^2 + |\tilde{b} \cdot k|^2 \)}}:

\smallskip
\noindent
\textit{For \(k \in S_1\)}: Since \(|k|\) is small and bounded away from zero, we have
\[
|\lambda_+|^2+|\tilde{b} \cdot k|^2 = 2|\tilde{b} \cdot k|^2\leq C .
\]

\smallskip
\noindent
\textit{For \(k \in S_2\cup S_3 \)}: From previous computations \eqref{S_2} and \eqref{S_3}, we obtain \(|\lambda_+| \sim |k|^2\). Then, it follows that
\[
|\lambda_+|^2 + |\tilde{b} \cdot k|^2 \leq C|k|^4.
\]

\medskip
\noindent
\underline{\textbf{Estimate of \( \left| \frac{\tilde{b} \cdot k}{\lambda_+}\right|^2 \)}}:

\smallskip
\noindent
\textit{For \(k \in S_1\)}: Since \( |\lambda_+|^2=|\tilde{b} \cdot k|^2   \), we have
\[
\left| \frac{\tilde{b} \cdot k}{\lambda_+} \right|^2 \leq C.
\]

\smallskip
\noindent
\textit{For \(k \in S_2\)}: Using \eqref{S_2}, \(|\lambda_+| \sim |k|^2\), and \( |\tilde{b} \cdot k| \lesssim C|k| \), we get
\[
\left| \frac{\tilde{b} \cdot k}{\lambda_+} \right| \leq C \frac{|k|}{|k|^2}  \leq C \frac{1}{|k|} \leq C.
\]

\smallskip
\noindent
\textit{For \(k \in S_3\)}: From \eqref{S_3}, \( |\lambda_+| \sim |k|^2 \),  we estimate
\[
\left| \frac{\tilde{b} \cdot k}{\lambda_+} \right| \leq \frac{|\tilde{b} \cdot k|}{|k|^2}.
\]

\medskip
\noindent
Combining the above, we obtain
\begin{equation}\label{eq:lambda-bounds}
|\lambda_+|^2 + |\tilde{b} \cdot k|^2 \leq
\begin{cases}
C, & k \in S_1, \\
C|k|^4, & k \in S_2 \cup S_3,
\end{cases}
\quad
\left| \frac{\tilde{b} \cdot k}{\lambda_+} \right|^2 \leq
\begin{cases}
C, & k \in S_1 \cup S_2, \\
\frac{|\tilde{b} \cdot k|^2}{|k|^4}, & k \in S_3.
\end{cases}
\end{equation}
Substituting the bounds for $|\widehat{G}(k,t)|$ obtained in \textbf{Step 2}, together with the estimates \eqref{eq:lambda-bounds}, into \eqref{K_12}, we immediately arrive at
\begin{equation*}
  \begin{cases}
    |\widehat{K_1}|,\, |\widehat{K_2}| \leq C e^{ -\frac{|k|^2}{8} t }, & \text{for } k \in S_1 \cup S_2, \\[6pt]
    |\widehat{K_1}| \leq C e^{ -\frac{(\tilde{b} \cdot k)^2}{|k|^2} t }, \quad
    |\widehat{K_2}| \leq C \frac{|\tilde{b} \cdot k|}{|k|^2} e^{ -\frac{(\tilde{b} \cdot k)^2}{|k|^2} t }, & \text{for } k \in S_3.
  \end{cases}
\end{equation*}
This completes the proof of Proposition \ref{pro2}.
\end{proof}

\section{Linear stability}\label{2/3}
In this section, we establish the linear stability and time decay estimates of system \eqref{equation}. For clarity, we restrict to the case $\mu=1$ and $\nu=0$, corresponding to Theorem \ref{thm1}. The complementary case $\mu=0$ and $\nu=1$, stated in Theorem \ref{thm1.1}, can be treated analogously and will therefore be omitted.
For the reader's convenience, we recall the linearized system:
\begin{equation}\label{linearrepeat}
\left\{ \begin{array}{l}
{\partial _t}U - \Delta U - \tilde{b} \cdot \nabla B = 0,\\[1ex]
{\partial _t}B - \tilde{b} \cdot \nabla U = 0,\\[1ex]
\mathrm{div}\,U = \mathrm{div}\,B = 0,\\[1ex]
(U,B)|_{t = 0} = (U_0, B_0).
\end{array} \right.
\end{equation}
Proposition \ref{pro1} provides an explicit representation of the Fourier transforms of $(U,B)$, from which we immediately obtain
\begin{align}\label{new1}
\begin{cases}
|\hat{U}|^2 \leq \big(|\widehat{K}_2|^2 + |\widehat{K}_3|^2\big) |\hat{\boldsymbol{\phi}}_0|^2, \\[1ex]
|\hat{B}|^2 \leq \big(|\widehat{K}_1|^2 + |\widehat{K}_3|^2\big) |\hat{\boldsymbol{\phi}}_0|^2,
\end{cases}
\end{align}
where $\boldsymbol{\phi}_0 := (U_0, B_0)^T$. The estimate \eqref{new1}, together with the kernel bounds established in Proposition \ref{pro2}, will be used in the proof of Theorem \ref{thm1} to derive decay properties for the Sobolev norms of $U$ and $B$.
\begin{proof}[Proof of Theorem {\rm\ref{thm1}}]
Due to \eqref{mean-zero}, for any $t\geq 0$, there holds $\int_{\mathbb{T}^n}U(x,t)\,d x=\int_{\mathbb{T}^n}B(x,t)\,d x=0$ and therefore
\begin{equation}\label{mean-zero2}
 \w{U}(0,t)=0,\quad  \w{B}(0,t)=0.
\end{equation}
Using \eqref{new1}, \eqref{mean-zero2}, and the Plancherel theorem, we deduce that for any $0 \le s \le m$,
\begin{align*}
\lVert B(x, t)\rVert^2_{\dot{H}^{s}}&=\sum_{k \in \mathbb{Z}^n \setminus \{0\}} |k|^{2s} |\widehat{B}(k, t)|^2\\
&\leq \sum_{k\neq 0}\lvert k\rvert^{2s}|\widehat{K_1}|^2|\hat{\boldsymbol{\phi}}_0|^2+
\sum_{k\neq 0}\lvert k\rvert^{2s}|\widehat{K_3}|^2|\hat{\boldsymbol{\phi}}_0|^2=:I_1+I_3,
\end{align*}
and
\begin{align}\label{U}
\lVert U(x, t)\rVert^2_{\dot{H}^{s}}\leq \sum_{k\neq 0}\lvert k\rvert^{2s}|\widehat{K_2}|^2|\hat{\boldsymbol{\phi}}_0|^2
+\sum_{k\neq 0}\lvert k\rvert^{2s}|\widehat{K_3}|^2|\hat{\boldsymbol{\phi}}_0|^2=:I_2+I_3.
\end{align}
 From Proposition \ref{pro2},  it is clear that
\begin{align}\label{25}
I_3= \sum_{k\neq 0}\lvert k\rvert^{2s}|\widehat{K_3}|^2|\hat{\boldsymbol{\phi}}_0|^2\leq  e^{-t}  \sum_{k\neq 0} \lvert k\rvert^{2s} \lvert \hat{\boldsymbol{\phi}}_0\rvert^{2}\leq e^{-t}\lVert{\boldsymbol{\phi}}_0\rVert^2_{\dot{H}^{s}}.
\end{align}
Utilizing Proposition \ref{pro2} and the Diophantine condition \eqref{Diophantine}, we have
\begin{align}
I_1= \sum_{k\neq 0}\lvert k\rvert^{2s}|\widehat{K_1}|^2|\hat{\boldsymbol{\phi}}_0|^2&\leq C\sum_{k\in S_1\cup S_2}\lvert k\rvert^{2s}e^{-\frac{|k|^2}{4}t}|\hat{\boldsymbol{\phi}}_0|^2+C\sum_{k\in S_3}e^{-\frac{2\lvert \tilde{b}\cdot k\rvert^{2}}{\lvert k\rvert^{2}}t} \lvert k\rvert^{2s} \lvert\hat{\boldsymbol{\phi}}_0\rvert^{2}\notag\\
&\leq Ce^{-\frac{t}{4}}\|{\boldsymbol{\phi}}_0(x)\|_{\dot{H}^{s}}^2+C\sum_{k\in S_3}e^{-\frac{2c^2}{|k|^{2+2r}} t}\lvert k\rvert^{2s} \lvert\hat{\boldsymbol{\phi}}_0\rvert^{2}.\notag
\end{align}
With the help of \eqref{ye}, the second term can be estimated as
\begin{align}
	&\sum_{k\in S_3}e^{-\frac{2c^2}{|k|^{2+2r}} t}\lvert k\rvert^{2s} |\hat{\boldsymbol{\phi}}_0|^2\notag\\
	=&\sum_{k\in S_3}e^{-\frac{2c^2}{|k|^{2+2r}} t}\left(\frac{t}{|k|^{2+2r}}\right)^{\frac{m-s}{r+1}}t^{-\frac{m-s}{1+r}}|k|^{{2(m -s)}}\lvert k\rvert^{2s} |\hat{\boldsymbol{\phi}}_0|^2\notag\\
	\leq&  t^{-\frac{m-s}{1+r}} \|{\boldsymbol{\phi}}_0(x)\|_{\dot{H}^{m}}^2\sup_{k\in S_3}e^{-\frac{2c^2}{|k|^{2+2r}} t}\left(\frac{t}{|k|^{2+2r}}\right)^{\frac{m-s}{r+1}}\notag\\
	\leq& Ct^{-\frac{m-s}{1+r}} \|{\boldsymbol{\phi}}_0(x)\|_{\dot{H}^{m}}^2,\notag
\end{align}	
and
\begin{align}
e^{-\frac{2c^2}{|k|^{2+2r}} t}\lvert k\rvert^{2s} |\hat{\boldsymbol{\phi}}_0|^2\leq \lvert k\rvert^{2s} |\hat{\boldsymbol{\phi}}_0|^2.\notag
\end{align}
It follows
\begin{align}
\sum_{k\in S_3}e^{-\frac{2c^2}{|k|^{2+2r}} t}\lvert k\rvert^{2s} |\hat{\boldsymbol{\phi}}_0|^2
	\leq& C(1+t)^{-\frac{m-s}{1+r}} \|{\boldsymbol{\phi}}_0(x)\|_{\dot{H}^{m}}^2.\notag
\end{align}
Combining the above estimates, we obtain, for any $0 \le s \le m$,
\begin{align}\label{24}
\| B(t) \|_{\dot{H}^s}^2
\le e^{-\frac{t}{4}} \, \| \boldsymbol{\phi}_0 \|_{\dot{H}^s}^2
+ C (1+t)^{-\frac{m-s}{r+1}} \, \| \boldsymbol{\phi}_0 \|_{\dot{H}^m}^2.
\end{align}
It remains to estimate \(I_2\), which is given by
\begin{align}
I_2=\sum_{k\neq 0}\lvert k\rvert^{2s}|\widehat{K_2}|^2|\hat{\boldsymbol{\phi}}_0|^2&\leq C\sum_{k\in S_1\cup S_2}\lvert k\rvert^{2s}e^{-\frac{|k|^2}{4}t}|\hat{\boldsymbol{\phi}}_0|^2+C\sum_{k\in S_3}\frac{|\tilde{b}\cdot k|^2}{|k|^{4}}e^{-\frac{2\lvert \tilde{b}\cdot k\rvert^{2}}{\lvert k\rvert^{2}}t} \lvert k\rvert^{2s} \lvert\hat{\boldsymbol{\phi}}_0\rvert^{2}.\notag
\end{align}
By virtue of \eqref{ye}, the second term admits the estimate
\begin{align}
	&\sum_{k\in S_3}\frac{|\tilde{b}\cdot k|^2}{|k|^{4}} e^{-\frac{2\lvert\tilde{b}\cdot k\rvert^{2}}{\lvert k\rvert^{2}}t} \lvert k\rvert^{2s} \lvert\hat{\boldsymbol{\phi}}_0\rvert^{2}\notag\\
	=&\sum_{k\in S_3}\frac{1}{\lvert k\rvert^{2}}\left({\frac{\lvert\tilde{b}\cdot k\rvert^{2}}{\lvert k\rvert^{2}}t}\right)t^{-1}
e^{-\frac{2\lvert \tilde{b}\cdot k\rvert^{2}}{\lvert k\rvert^{2}}t}\left( {\frac{\lvert\tilde{b}\cdot k\rvert^{2}}{\lvert k\rvert^{2}}t} \right)^{\frac{m-s+1}{r+1}}\left( {\frac{\lvert\tilde{b}\cdot k\rvert^{2}}{\lvert k\rvert^{2}}t} \right)^{-\frac{m-s+1}{r+1}}
\lvert k\rvert^{2s} \lvert\hat{\boldsymbol{\phi}}_0\rvert^{2}\notag\\
	\leq&\sum_{k\in S_3}\frac{1}{\lvert k\rvert^{2}}\left({\frac{\lvert\tilde{b}\cdot k\rvert^{2}}{\lvert k\rvert^{2}}t}\right)t^{-1}
e^{-\frac{2\lvert \tilde{b}\cdot k\rvert^{2}}{\lvert k\rvert^{2}}t}\left( {\frac{\lvert\tilde{b}\cdot k\rvert^{2}}{\lvert k\rvert^{2}}t} \right)^{\frac{m-s+1}{r+1}}t^{-\frac{m-s+1}{1+r}}|k|^{(2+2r)\frac{m-s+1}{1+r}}
\lvert k\rvert^{2s} \lvert\hat{\boldsymbol{\phi}}_0\rvert^{2}\notag\\
=&\sum_{k\in S_3} t^{-\left(1+\frac{m-s+1}{1+r}\right)}e^{-\frac{2\lvert \tilde{b}\cdot k\rvert^{2}}{\lvert k\rvert^{2}}t}\left( {\frac{\lvert\tilde{b}\cdot k\rvert^{2}}{\lvert k\rvert^{2}}t} \right)^{\left(1+\frac{m-s+1}{1+r}\right)}
\lvert k\rvert^{2m}\lvert\hat{\boldsymbol{\phi}}_0\rvert^{2}\notag\\
	\leq&  t^{-\left(1+\frac{m-s+1}{1+r}\right)} \|{\boldsymbol{\phi}}_0(x)\|_{\dot{H}^{m}}^2\sup_{k\in S_3}e^{-\frac{2\lvert \tilde{b}\cdot k\rvert^{2}}{\lvert k\rvert^{2}}t}\left( {\frac{\lvert\tilde{b}\cdot k\rvert^{2}}{\lvert k\rvert^{2}}t} \right)^{\left(1+\frac{m-s+1}{1+r}\right)}\notag\\
	\leq& Ct^{-\left(1+\frac{m-s+1}{1+r}\right)} \|{\boldsymbol{\phi}}_0(x)\|_{\dot{H}^{m}}^2.\notag
\end{align}
As $\tilde{b}$ is fixed, it follows that
\begin{align}
\frac{|\tilde{b}\cdot k|^2}{|k|^{4}} e^{-\frac{2\lvert\tilde{b}\cdot k\rvert^{2}}{\lvert k\rvert^{2}}t} \lvert k\rvert^{2s} \lvert\hat{\boldsymbol{\phi}}_0\rvert^{2}\leq \lvert k\rvert^{2s} |\hat{\boldsymbol{\phi}}_0|^2.\notag
\end{align}
Thus
\begin{align*}
\sum_{k\in S_3}\frac{|\tilde{b}\cdot k|^2}{|k|^{4}} e^{-\frac{2|\tilde{b}\cdot k|^2}{|k|^{2}}t} |k|^{2s} \lvert\hat{\boldsymbol{\phi}}_0\rvert^{2}&\leq C (1 + t)^{-\left(1+\frac{m-s+1}{(1 + r)}\right)} \|{\boldsymbol{\phi}}_0(x)\|_{\dot{H}^{m}}^2.
\end{align*}
Hence, we arrive at, for any $0 \le s \le m$,
\begin{align*}
\| U(t) \|_{\dot{H}^s}^2
\le e^{-\frac{t}{4}} \| \boldsymbol{\phi}_0 \|_{\dot{H}^s}^2
+ C (1 + t)^{-\left(1 + \frac{m-s+1}{1+r}\right)} \| \boldsymbol{\phi}_0 \|_{\dot{H}^m}^2,
\end{align*}
which, together with the estimate for $B(t)$ in \eqref{24}, completes the proof of Theorem \ref{thm1}.
\end{proof}

\section{Proof of Theorem \ref{thm.2}}\label{s:7}
In this section, we establish the nonlinear stability result stated in Theorem \ref{thm.2}. The analysis is divided into two subsections.
\subsection{Global-in-time existence}\label{4/2a}
The local well-posedness of system \eqref{equation} can be established by standard methods, such as the Friedrichs mollifier or Fourier cutoff techniques; see, for instance, \cite{li-2017-local,fefferman-2014-local,Majda-Bertozzi}. Therefore, it suffices to derive uniform \emph{a priori} bounds to extend the local solution globally in time.

We start with the following energy inequality, which can be found in Lemma 4.1 of \cite{XieJiu-CVPDE2024}. The proof is omitted for brevity.
\begin{prop}\label{prop3.2}
Let $(u,b)$ be a smooth solution to \eqref{equation} with $\mu =0$ and $\nu = 1$.
Then for any $s \in [0,m]$ and $t \in [0,T]$, the following energy inequality holds:
\begin{align}\label{energy1'}
\frac{1}{2}\frac{d}{dt} \lVert (u, b)(t) \rVert_{H^s}^2
+ \lVert \nabla b(t) \rVert_{H^s}^2
\;\leq\; C \Big( \lVert \nabla u(t) \rVert_{L^\infty}
+ \lVert \nabla b(t) \rVert_{L^\infty} \Big)
\lVert (u, b)(t) \rVert_{H^s}^2 .
\end{align}
\end{prop}

\medskip

We now turn to the proof of the global existence part of Theorem \ref{thm.2}.
Taking $s = m$ in \eqref{energy1'} and integrating over $[0,T]$ yields
\begin{align}\label{energy-m}
&\|(u,b)(T)\|_{H^m}^2+\int_0^\infty \lVert \nabla b(t) \rVert_{H^m}^2 \,{d}t \notag\\
\leq& \|(u_0,b_0)\|_{H^m}^2
+ C \sup_{t\in[0,T]}\|(u,b)(t)\|_{H^m}^2
\int_0^T \big( \|\nabla u(t)\|_{L^\infty} + \|\nabla b(t)\|_{L^\infty} \big)\,dt.
\end{align}
To rule out possible blow-up and extend the solution globally, it remains to estimate the critical quantity in \eqref{energy-m}, namely,
\[
\int_0^T \big(\|\nabla u(t)\|_{L^\infty}+\|\nabla b(t)\|_{L^\infty}\big)\,dt.
\]
Using the Fourier series representation,  we obtain for  $t \ge 0$,
\[
\|\nabla u(t)\|_{L^\infty}
\leq C\sum_{k\in \mathbb{Z}^n\setminus\{0\}} |k|\,|\hat{u}(t,k)|,
\qquad
\|\nabla b(t)\|_{L^\infty}
\leq C\sum_{k\in \mathbb{Z}^n\setminus\{0\}} |k|\,|\hat{b}(t,k)|,
\]
so that
\begin{align*}
&\int_0^T \big(\|\nabla u(t)\|_{L^\infty} + \|\nabla b(t)\|_{L^\infty}\big)\,dt\\
\leq &C\sum_{k\in \mathbb{Z}^n\setminus\{0\}}\int_0^T |k|\,|\hat u(t,k)|\,dt
+ C\sum_{k\in \mathbb{Z}^n\setminus\{0\}}\int_0^T |k|\,|\hat b(t,k)|\,dt.
\end{align*}

Claim that
\begin{align}\label{prop-applied}
 &\sum_{k \in \mathbb{Z}^n \setminus \{0\}} \int_{0}^{T} |k|\left| \hat{u}(t,k) \right| \, dt
+ \sum_{k \in \mathbb{Z}^n \setminus \{0\}} \int_{0}^{T} |k|| \hat{b}(t,k)| \, dt\leq C\|(u_0,b_0)\|_{H^m} \notag\\
&\quad
+ C_1 \sup_{t \in [0,T]} \|(u, b)(t)\|_{H^m}
\left( \sum_{k \in \mathbb{Z}^n \setminus \{0\}}\int_0^T |k|\,|\hat u(t,k)| \, dt
+ \sum_{k \in \mathbb{Z}^n \setminus \{0\}}\int_0^T |k|\,|\hat b(t,k)| \, dt \right).
\end{align}
for some constant \(C_1>0\).

By choosing the initial data sufficiently small as in \eqref{smallcondition}, the last term on the right-hand side of \eqref{prop-applied} can be absorbed into the left-hand side via a standard bootstrap argument. Hence we obtain
\begin{align}\label{Linfty}
\int_0^T \big(\|\nabla u(t)\|_{L^\infty} + \|\nabla b(t)\|_{L^\infty}\big)\,dt
\leq C \|(u_0,b_0)\|_{H^m}, \qquad \forall\, T>0.
\end{align}
Substituting \eqref{Linfty}  into \eqref{energy-m} and using the smallness assumption \eqref{smallcondition} again, we deduce
\begin{align}\label{Linftyab}
\sup_{t\in[0,\infty)} \|(u,b)(t)\|_{H^m}^2+\int_0^\infty \lVert \nabla b(t) \rVert_{H^m}^2 \,{d}t
\leq C \|(u_0,b_0)\|_{H^m}^2,
\end{align}
which establishes the global-in-time existence of smooth solutions and the uniform bound \eqref{eq:1.2a}.

Finally, we justify the Fourier-side control claimed in \eqref{prop-applied}.
\begin{prop}\label{prop6.1a}
Let $n \geq 2$ and $m > 4 + 2r + \tfrac{n}{2}$.
Assume that $(u,b)$ is a smooth solution to \eqref{equation} with $\mu=0$ and $\nu=1$.
Then there exists a constant $C>0$ such that for all $T>0$,
\begin{align}\label{prop6.11a}
 &\sum_{k \in \mathbb{Z}^n \setminus \{0\}} \int_{0}^{T} |k|\left| \hat{u}(t,k) \right| \, dt
+ \sum_{k \in \mathbb{Z}^n \setminus \{0\}} \int_{0}^{T} |k|| \hat{b}(t,k)| \, dt\leq C\|(u_0,b_0)\|_{H^m} \notag\\
&\quad
+ C \sup_{t \in [0,T]} \|(u, b)(t)\|_{H^m}
\left( \sum_{k \in \mathbb{Z}^n \setminus \{0\}}\int_0^T |k|\,|\hat u(t,k)| \, dt
+ \sum_{k \in \mathbb{Z}^n \setminus \{0\}}\int_0^T |k|\,|\hat b(t,k)| \, dt \right).
\end{align}
\end{prop}
\begin{proof}
Applying Duhamel's principle to $\eqref{equation}_2$, we obtain
\begin{align}\label{le'a}
&\sum_{k \in \mathbb{Z}^n \setminus \{0\}} \int_{0}^{T} |k| |\hat{b}(t,k)| \, dt \leq I_{6} + I_{7} + I_{8} + I_{9},
\end{align}
where
\begin{align*}
&I_{6}:= \sum_{k \in \mathbb{Z}^n \setminus \{0\}} \int_{0}^{T} |k| e^{-|k|^2 t} |\hat{b}_0(k)| \, dt,\notag\\[1ex]
&I_{7}:= \sum_{k \in \mathbb{Z}^n \setminus \{0\}} \int_{0}^{T} \int_{0}^{t} |k| e^{-|k|^2 (t - \tau)} |\widehat{u \cdot \nabla b}(\tau,k)| \, d\tau dt,\notag\\[1ex]
&I_{8} := \sum_{k \in \mathbb{Z}^n \setminus \{0\}} \int_{0}^{T} \int_{0}^{t} |k| e^{-|k|^2 (t - \tau)} |\widehat{b \cdot \nabla u} (\tau,k)| \, d\tau dt,\notag\\[1ex]
&I_{9}:= \sum_{k \in \mathbb{Z}^n \setminus \{0\}} \int_{0}^{T} \int_{0}^{t} |k| e^{-|k|^2 (t - \tau)} |(\tilde{b} \cdot k) \hat{u} (\tau,k)| \, d\tau dt.
\end{align*}
It is straightforward to estimate $I_6$. Indeed, we have
$$
I_6 \leq \sum_{k \in \mathbb{Z}^n \setminus \{0\}} \int_0^T |k|^2 e^{-|k|^2 t} |\hat{b}_0| \, dt \leq \sum_{k \in \mathbb{Z}^n \setminus \{0\}} |\hat{b}_0(k)|,
$$
where we used the elementary inequality
\begin{align}\label{2.3}
\int_{0}^{T} a e^{-a t} dt\leq 1   \quad {\rm for  \quad all } \quad a\geq0,\, T\geq0.
\end{align}
Applying H\"{o}lder inequality and noting that $m > \frac{n}{2}$, we further deduce
\begin{align*}
I_6 \leq \sum_{k \in \mathbb{Z}^n \setminus \{0\}} |\hat{b}_0(k)|\leq \left( \sum_{k \in \mathbb{Z}^n \setminus \{0\}} |k|^{-2m} \right)^{\frac{1}{2}} \left( \sum_{k \in \mathbb{Z}^n \setminus \{0\}} |k|^{2m} |\hat{b}_0|^2 \right)^{\frac{1}{2}} \leq C \left\lVert b_0 \right\rVert_{\dot{H}^m},
\end{align*}
where $C > 0$ is a constant depending only on $n$ and $m$.

We now turn to the estimate of $I_{7}$. Since $\nabla \cdot u = 0$, we may rewrite
\[
|\widehat{u \cdot \nabla b}(t,k)|
= \big|\widehat{\nabla \cdot (u \otimes b)}(t,k)\big|
\leq |k|\,\big|(\hat{u} * \hat{b})(t,k)\big|,
\quad k \in \mathbb{Z}^n \setminus \{0\}.
\]
Substituting the above bound into the definition of $I_7$, and then applying Young inequality together with \eqref{2.3}, we obtain
\begin{align*}
I_{7} &\leq C\sum_{k \in \mathbb{Z}^n \setminus \{0\}} \int_{0}^{T} \int_{0}^{t} |k| e^{-|k|^2 (t - \tau)} |k|\, \left|(\hat{u}*\hat{b})(\tau,k)\right| \, d\tau dt\\[1ex]
&\leq C \sum_{k \in \mathbb{Z}^n \setminus \{0\}} \int_{0}^{T} \left|(\hat{u}*\hat{b}) (t,k)\right| \, d t.
\end{align*}
Next, applying Young inequality for convolutions in the spatial frequency variable \(k \in \mathbb{Z}^n\), and subsequently using H\"{o}lder inequality, we get
\begin{align*}
I_{7}&\leq C\int_{0}^{T} \sum_{k \in \mathbb{Z}^n \setminus \{0\}}  \left|(\hat{u}*\hat{b}) (t,k)\right| \, d t
\leq \int_{0}^{T} \sum_{k \in \mathbb{Z}^n \setminus \{0\}} |\hat{u} (t,k)| \sum_{k \in \mathbb{Z}^n \setminus \{0\}} |\hat{b} (t,k)| \, dt\\[1ex]
&\leq C \sup_{t \in [0, T]} \sum_{k \in \mathbb{Z}^n \setminus \{0\}} |\hat{u} (t,k)| \left( \sum_{k \in \mathbb{Z}^n \setminus \{0\}} \int_{0}^{T} |\hat{b} (t,k)| \, dt \right)\\[1ex]
&\leq C \sup_{t \in [0, T]} \| u(t) \|_{H^m} \left( \sum_{k \in \mathbb{Z}^n \setminus \{0\}} \int_{0}^{T} |k| |\hat{b} (t,k)| \, dt \right),
\end{align*}
where we have used the Poincar\'e inequality in the last step since \(k \ne 0\).

Similarly, for $I_{8}$, we have
\begin{align*}
I_{8}\leq C \int_{0}^{T}  \sum_{k \in \mathbb{Z}^n \setminus \{0\}}\left| (\hat{b}*\hat{u}) (t,k)\right| \, dt\leq C \sup_{t \in [0, T]} \left\lVert u  \right\rVert_{H^m} \left( \sum_{k \in \mathbb{Z}^n \setminus \{0\}} \int_{0}^{T} |k||\hat{b} (t,k)| \, dt \right).
 \end{align*}

For $I_9$,  we can see by Fubini's theorem and \eqref{2.3} that
\begin{align*}
I_9 \leq C(|\tilde{b}|) \sum_{k \in \mathbb{Z}^n \setminus \{0\}} \int_{0}^{T} \left| \hat{u}(t,k) \right| \, dt \leq C(|\tilde{b}|) \sum_{k \in \mathbb{Z}^n \setminus \{0\}} \int_{0}^{T} |k| \left|\hat{u}(t,k) \right| \, dt.
\end{align*}
Recall $\eqref{2.4s}_1$ and write
\begin{align}
&\sum_{k \in \mathbb{Z}^n \setminus \{0\}} \int_{0}^{T} |k||\hat{u}(t,k)| \, dt\notag\\[1ex]
\leq& \sum_{k \in \mathbb{Z}^n \setminus \{0\}} \int_{0}^{T} |k||\widehat{K_1}||\hat{\boldsymbol{\psi}}_0|\, dt+\sum_{k \in \mathbb{Z}^n \setminus \{0\}} \int_{0}^{T} \int_{0}^{t}
|k||\widehat{K_1}(t-\tau)||\widehat{N}(\tau)| \, d\tau dt\notag\\[1ex]
&+\sum_{k \in \mathbb{Z}^n \setminus \{0\}} \int_{0}^{T} |k||\widehat{K_3}||\hat{\boldsymbol{\psi}}_0|  \, dt+\sum_{k \in \mathbb{Z}^n \setminus \{0\}} \int_{0}^{T} \int_{0}^{t}
|k||\widehat{K_3}(t-\tau)||\widehat{N_1}(\tau)| \, d\tau dt\notag\\[1ex]
=:&I_{91} + I_{92}+ I_{93} + I_{94}.\label{le'''}
\end{align}
We begin by estimating \(I_{93}\). Proposition \ref{pro2} and \eqref{2.3} give
\begin{align*}
I_{93} \leq \sum_{k \in \mathbb{Z}^n \setminus \{0\}} \int_0^T |k| e^{-\frac{|k|^2}{2} t} |\hat{\boldsymbol{\psi}}_0(k)| \, dt \leq C \| \boldsymbol{\psi}_0 \|_{H^m}.
\end{align*}
For the term \(I_{94}\), we apply Proposition \ref{pro2} and proceed similarly to the estimate of \(I_{7}\), which yields
\begin{align*}
I_{94}
&\leq C \int_0^T \sum_{k \in \mathbb{Z}^n \setminus \{0\}}
\left( |(\hat{u} * \hat{u})(t,k)| + |(\hat{b} * \hat{b})(t,k)| \right) \, dt \\[1ex]
&\leq C \sup_{t \in [0, T]} \| (u, b)(t) \|_{H^m}\left( \int_0^T \sum_k |k| |\hat{u}(t,k)| \, dt + \int_0^T \sum_k |k||\hat{b}(t,k)| \, dt \right).
\end{align*}
We now estimate \(I_{91}\). By Proposition \ref{pro2}, the Diophantine condition, and the Poincar\'e inequality, we derive
\begin{align*}
I_{91}
&\leq \sum_{k \in S_1 \cup S_2} \int_0^T |k||\widehat{K_1}(t,k)|\, |\hat{\boldsymbol{\psi}}_0(k)| \, dt
+ \sum_{k \in S_3} \int_0^T|k| |\widehat{K_1}(t,k)|\, |\hat{\boldsymbol{\psi}}_0(k)| \, dt \\[1ex]
&\leq C \sum_{k \in S_1 \cup S_2} \int_0^T |k| e^{-\frac{|k|^2}{8} t} |\hat{\boldsymbol{\psi}}_0(k)| \, dt
+ C \sum_{k \in S_3} \int_0^T |k| e^{-\frac{|\tilde{b}\cdot k|^2}{|k|^2} t} |\hat{\boldsymbol{\psi}}_0(k)| \, dt  \\[1ex]
&\leq C \sum_{k \in S_1 \cup S_2} |\hat{\boldsymbol{\psi}}_0(k)|
+ C \sum_{k \in S_3} \frac{|k|^3}{|\tilde{b} \cdot k|^2} \left( \int_0^T \frac{|\tilde{b} \cdot k|^2}{|k|^2} e^{-\frac{|\tilde{b} \cdot k|^2}{|k|^2} t} dt \right) |\hat{\boldsymbol{\psi}}_0(k)|\\[1ex]
&\leq C \sum_{k \in \mathbb{Z}^n \setminus \{0\}} |\hat{\boldsymbol{\psi}}_0(k)|
+ C \sum_{k \in \mathbb{Z}^n \setminus \{0\}} |k|^{3+2r} |\hat{\boldsymbol{\psi}}_0(k)| \\[1ex]
&\leq C \| \boldsymbol{\psi}_0 \|_{H^m}.
\end{align*}
Finally, to estimate $I_{92}$, we employ Proposition \ref{pro2} in conjunction with \eqref{1.411}, \eqref{1.412}, and then apply Young, Poincar\'e, and H\"older inequalities, which yield for $m > 4 + 2r + \tfrac{n}{2}$
\begin{align*}
I_{92}
&\leq  \sum_{k \in S_1 \cup S_2} \int_0^T \int_0^t|k| |\widehat{K_1}(t-\tau,k)|\, |\widehat{N}(\tau,k)| \, d\tau dt\\[1ex]
&\quad+ \sum_{k \in S_3} \int_0^T|k| \int_0^t |\widehat{K_1}(t-\tau,k)|\, |\widehat{N}(\tau,k)| \, d\tau dt \\[1ex]
&\leq C \sum_{k \in S_1 \cup S_2} \int_0^T \int_0^t |k|e^{-\frac{|k|^2}{8}(t-\tau)} |\widehat{N}(\tau,k)| \, d\tau dt
+ C \sum_{k \in S_3} \int_0^T |k|^{3+2r} |\widehat{N}(t,k)| \, dt \\[1ex]
&\leq C \sum_{k \in S_1 \cup S_2} \int_0^T \int_0^t e^{-\frac{|k|^2}{8}(t-\tau)} |k|
\left( |\hat{u} * \hat{u}| + |\hat{b} * \hat{b}| + |\hat{u} * \hat{b}| \right)(\tau,k) \, d\tau dt \\[1ex]
&\quad + C \int_0^T \sum_{k \in S_3} |k|^{4+2r}
\left( |\hat{u} * \hat{u}| + |\hat{b} * \hat{b}| + |\hat{u} * \hat{b}| \right)(t,k) \, dt \\[1ex]
&\leq C \int_0^T \sum_{k \in \mathbb{Z}^n \setminus \{0\}} |k|^{4+2r}
\left( |\hat{u} * \hat{u}| + |\hat{b} * \hat{b}| + |\hat{u} * \hat{b}| \right)(t,k) \, dt \\[1ex]
&\leq C \sup_{t \in [0, T]} \| (u, b)(t) \|_{H^m} \left( \int_0^T \sum_{k \in \mathbb{Z}^n \setminus \{0\}} |k| |\hat{u}(t,k)| \, dt + \int_0^T \sum_{k \in \mathbb{Z}^n \setminus \{0\}} |k||\hat{b}(t,k)| \, dt \right).
\end{align*}
Combining the above estimate into \eqref{le'''}, which completes the proof of Proposition \ref{prop6.1a}.
\end{proof}

\subsection{Temporal decay estimates}\label{4/3}
In this subsection, we establish the temporal decay part of Theorem \ref{thm.2}.
We first focus on deriving the temporal decay of $\|(u, b)(t)\|_{L^2(\mathbb{T}^n)}$.
Once the $L^2$ decay is established, higher-order norms can be controlled using standard interpolation inequalities together with the uniform bound \eqref{Linftyab}, which then yields the decay estimate stated in \eqref{finaldecay111}.
\begin{prop}\label{prop6.6}
Let $n\geq 2$ and $m> 4 + 2r + \frac{n}{2}$ with $r>n-1$. Let $(u, b)$ be a smooth global-in-time solution to \eqref{equation} with $\mu =0$ and $\nu = 1$. Then there exists a constant $C > 0$ such that
\begin{align*}
\|u\|_{L^2(\mathbb{T}^n)}^2+\|b\|_{L^2(\mathbb{T}^n)}^2\leq  C (1+t)^{-\frac{m}{1+r}}.
\end{align*}
\end{prop}

\begin{proof}
Since $\nabla\cdot u=\nabla \cdot b=0$, the $L^2$ energy identity reads
\begin{align}\label{L2}
\frac{d}{dt}\left(\| u(t)\|_{L^2}^2+\| b(t)\|_{L^2}^2\right) = -2\|\nabla b(t)\|_{L^2}^2.
\end{align}
To capture additional dissipation, we will make use of the following cross-term estimate:
\begin{align}\label{key1b-intro}
-\frac{d}{dt}\int_{\mathbb{T}^n} (\tilde{b} \cdot \nabla u) \cdot \Lambda^{-2} b \, dx
&\leq \left( \frac{1}{2} + |\tilde{b}|^2 \right) \|\nabla b\|_{L^2}^2
- \frac{1}{2} \|\Lambda^{-1}(\tilde{b}\cdot\nabla u)\|_{L^2}^2 \notag\\
&\quad + C\left( \|\nabla b\|_{L^\infty} + \|\nabla u\|_{L^\infty} \right) \left( \|b\|_{L^2}^2 + \|u\|_{L^2}^2 \right),
\end{align}
whose precise statement and proof are deferred to Proposition \ref{prop3.21} at the end of this subsection.

Multiplying \eqref{L2} by $A := \frac{1}{2}(1+|\tilde{b}|+|\tilde{b}|^2)$ and adding the result to \eqref{key1b-intro}, we arrive at
\begin{align}
	&\frac{d}{dt}\left(A\| (u, b)\|_{L^2}^2-\int_{\mathbb{T}^n}\left(\tb\cdot\nabla u\right)\cdot \Lambda^{-2} b\, dx\right)\notag\\\leq &-\frac{1}{2}\left(\|\nabla b\|_{L^2}^2+\|\Lambda^{-1}(\tb\cdot\nabla u)\|_{L^2}^2\right)+
 C\left(\|\nabla u\|_{L^\infty}+\|\nabla b\|_{L^\infty}\right)\|(u, b)\|_{L^2}^2.\label{energy2}
\end{align}
On the other hand, applying Lemmas \ref{2.1}, \ref{lem:zero-mean-fractional} and Poincar\'e inequality  yield
\begin{align*}
\|\Lambda^{-1-r} u\|_{L^2}
\leq c\,\|\tilde{b}\cdot\nabla(\Lambda^{-1-r}u)\|_{H^r}
\leq c\,\|\Lambda^{-1-r}(\tilde{b}\cdot\nabla u)\|_{\dot{H}^r} = c\,\|\Lambda^{-1}(\tilde{b}\cdot\nabla u)\|_{L^2}.
\end{align*}
By Poincar\'e inequality and \eqref{meanzero2}, we have
\[
\|b\|_{L^2} \leq \|\nabla b\|_{L^2}.
\]
Consequently, the differential inequality \eqref{energy2} can be refined as
\begin{align}\label{6.4a}
&\frac{d}{dt} \left(A\left(\| (u, b)\|_{L^2}^2\right)-\int_{\mathbb{T}^n}\left(\tb\cdot\nabla u\right)\cdot \Lambda^{-2} b\, dx\right)\notag\\
\leq& -\frac{c^*}{4} \left( \left\lVert b \right\rVert_{L^2}^2 + \left\lVert \Lambda^{-1-r} u \right\rVert_{L^2}^2 \right)
+C\left(\|\nabla u\|_{L^\infty}+\|\nabla b\|_{L^\infty}\right)\|(u, b)\|_{L^2}^2,
\end{align}
where $c^* := \min\{1,c\}$ and $c$ is the constant appearing in \eqref{Diophantine}.

Let \(M>0\) be a parameter to be fixed later such that \(\tfrac{A}{M}\leq 1\). By Plancherel's theorem,
\begin{align*}
	&\frac{A}{M}\| u\|_{L^2}^2-\|\Lambda^{-1-r} u\|_{L^2}^2=\sum_{|k|\neq 0}\left(\frac{A}{M}-|k|^{-2-2r}\right)\left|\hat{u}(k)\right|^2\\
	\leq& \frac{A}{M}\sum_{\frac{A}{M}\geq |k|^{-2r-2}}\left|\hat{u}(k)\right|^2\notag\\
	\leq& \left(\frac{A}{M}\right)^{1+\frac{m}{1+r}}\sum_{\frac{A}{M}\geq |k|^{-2r-2}}\left(\dfrac{M}{A}\right)^{\frac{m}{1+r}}
\left(\frac{1}{|k|}\right)^{2m}|k|^{2m}\left|\hat{u}(k)\right|^2\notag\\
	\leq& \left(\frac{A}{M}\right)^{1+\frac{m}{1+r}}\|\Lambda^{m}u\|^2_{L^2}.
\end{align*}
Moreover, H\"older, Poincar\'e and  Young inequalities yield
\begin{align}\label{916}
\left| \int_{\mathbb{T}^n} (\tilde{b} \cdot \nabla u) \cdot \Lambda^{-2} b \, dx \right|
\leq |\tilde{b}| \cdot \lVert u \rVert_{L^2} \cdot \lVert \Lambda^{- 1} b \rVert_{L^2} \leq \frac{|\tilde{b}|}{2} \left( \lVert  u \rVert_{L^2}^2 + \lVert  b \rVert_{L^2}^2 \right).
\end{align}
Then, the first term on the right-hand side of \eqref{6.4a} turns to
\begin{align}\label{6.4aa}
&- \frac{c^{*}}{4} \left( \left\lVert b \right\rVert_{L^2}^2 + \left\lVert \Lambda^{-1-r} u \right\rVert_{L^2}^2 \right)\notag\\
= &- \frac{c^{*}}{4} \left\lVert b \right\rVert_{L^2}^2 - \frac{c^* A}{4M} \left\lVert u \right\rVert_{L^2}^2
+ \frac{c^*}{4} \left( \frac{A}{M} \left\lVert  u \right\rVert_{L^2}^2 - \left\lVert \Lambda^{-1 - r} u \right\rVert_{L^2}^2 \right) \notag \\
\leq& - \frac{c^* A}{4 M} \left( \left\lVert b \right\rVert_{L^2}^2 + \left\lVert u \right\rVert_{L^2}^2 \right)
+\frac{C}{M^{1+\frac{m}{1+r}}}\|\Lambda^{m}u\|^2_{L^2}\notag \\
\leq &- \frac{c^*}{8 M}F(t)
+ \frac{C}{M^{1+\frac{m}{1+r}}}\|\Lambda^{m}u\|^2_{L^2},
\end{align}
where we set
\[
F(t):= A\|(u,b)\|_{L^2}^2 - \int_{\mathbb{T}^n}(\tb\cdot\nabla u)\cdot \Lambda^{-2}b\,dx.
\]
Substituting   \eqref{6.4aa} into \eqref{6.4a} gives
\begin{align}\label{6.6a}
\frac{d}{dt} F(t)
\leq - \frac{c^*}{8 M } F(t)
+\frac{C}{M^{1+\frac{m}{1+r}}}\|\Lambda^{m}u\|^2_{L^2}
+ C\left(\|\nabla u\|_{L^\infty}+\|\nabla b\|_{L^\infty}\right)\|(u, b)\|_{L^2}^2.
\end{align}
Choosing
\[
M=A+\frac{c^*t}{8j},\qquad j>\frac{m}{1+r},
\]
and multiplying \eqref{6.6a} by \(M^j\) together with \eqref{916} yields
\begin{align*}
\frac{d}{dt}\big(M^jF(t)\big)
\leq C M^{j-1-\frac{m}{1+r}}\|\Lambda^m u\|_{L^2}^2
+C M^jF(t)(\|\nabla u\|_{L^\infty}+\|\nabla b\|_{L^\infty}).
\end{align*}
Integrating over \([0,t]\), we deduce
\begin{align}
M^jF(t)\leq &A^jF(0)+C\int_{0}^tM^{j-1-\frac{m}{1+r}}\|\Lambda^{m}u\|^2_{L^2}\, d \tau\notag\\
&+C\int_{0}^tM^jF(\tau)\left(\|\nabla u(\tau)\|_{L^\infty}+\|\nabla b(\tau)\|_{L^\infty}\right)\, d\tau.\label{gp.5.25-opt}
\end{align}
Multiplying \eqref{gp.5.25-opt} by \((1+t)^{\frac{m}{1+r}-j}\) and using  \eqref{Linftyab}, we obtain
\begin{align*}
	&(1+t)^{\frac{m}{1+r}}F(t)\notag\\
\leq &C(1+t)^{\frac{m}{1+r}-j}A^jF(0)+C(1+t)^{\frac{m}{1+r}-j}
\int_{0}^t(A+\dfrac{c^*\tau}{8j})^{j-1-\frac{m}{1+r}}\|\Lambda^{m}u\|^2_{L^2}\, d \tau\notag\\&+C(1+t)^{\frac{m}{1+r}-j}\int_{0}^t(A+\dfrac{c^*\tau}{8j})^jF(\tau)\left(\|\nabla u(\tau)\|_{L^\infty}+\|\nabla b(\tau)\|_{L^\infty}\right)\, d\tau\notag\\
\leq &C+C(1+t)^{\frac{m}{1+r}-j+j-\frac{m}{1+r}}\sup_{0\leq\tau\leq t}\|\Lambda^{m}u(\tau)\|_{L^2}^2\notag\\&+C(1+t)^{\frac{m}{1+r}-j}(1+t)^{j-\frac{m}{1+r}}\int_{0}^t(1+\tau)^{\frac{m}{1+r}}F(\tau)\left(\|\nabla u(\tau)\|_{L^\infty}+\|\nabla b(\tau)\|_{L^\infty}\right)\, d\tau\notag\\\leq &C+C\int_{0}^t(1+\tau)^{\frac{m}{1+r}}F(\tau)\left(\|\nabla u(\tau)\|_{L^\infty}+\|\nabla b(\tau)\|_{L^\infty}\right)\, d\tau.
\end{align*}
Moreover, \eqref{Linfty} yields
\[
\int_0^t(\|\nabla u(\tau)\|_{L^\infty}+\|\nabla b(\tau)\|_{L^\infty})\,d\tau\leq C.
\]
It then follows from Gr\"onwall's inequality that
\[
(1+t)^{\frac{m}{1+r}}F(t)\leq C.
\]
Recalling the definition of \(F(t)\) and using \eqref{916}, we have
\[
F(t)\geq \frac{1}{2}\|(u,b)(t)\|_{L^2}^2.
\]
Consequently,
\begin{align}\label{final-L2}
\|u(t)\|_{L^2}^2+\|b(t)\|_{L^2}^2 \leq C(1+t)^{-\frac{m}{1+r}}.
\end{align}
This concludes the proof of Proposition \ref{prop6.6}.
\end{proof}
Finally, interpolation between \eqref{final-L2} and the uniform \(H^m\)-bound \eqref{Linftyab} yields, for any \(s\in[0,m]\),
\[
\|u(t)\|_{H^s} + \|b(t)\|_{H^s} \leq C (1+t)^{-\frac{m-s}{2r+2}}.
\]
Hence, Theorem \ref{thm.2} follows.

To conclude this subsection, we provide an auxiliary estimate \eqref{key1b-intro} that was used in the proof of Proposition \ref{prop6.6}.
\begin{prop}\label{prop3.21}
Let $(u,b)$ be a smooth global solution to $\eqref{equation}$ with $\mu = 0$ and $\nu = 1$. Then
\begin{align}\label{key1b}
	-\frac{d}{d t}\int_{\mathbb{T}^n}\left(\tb\cdot\nabla u\right)\cdot \Lambda^{-2} b\,dx&\leq \left( \frac{1}{2} + |\tilde{b}|^2 \right)\left\|\nabla b\right\|^2_{L^2}
-\frac{1}{2}\left\|\Lambda^{-1}\left(\tb\cdot\nabla u\right)\right\|^2_{L^2} \notag\\
	&\quad+C\left(\|\nabla b\|_{L^\infty}+\|\nabla u\|_{L^\infty}\right) \left(\| b\|_{L^2}^2+\| u\|_{L^2}^2\right).
\end{align}
\end{prop}
\begin{proof}
By straightforward calculations, we have
\begin{align}\label{J12}
-\frac{d}{dt} \int_{\mathbb{T}^n} (\tilde{b} \cdot \nabla u) \cdot \Lambda^{-2} b \, dx &= -\int_{\mathbb{T}^n} (\tilde{b} \cdot \nabla \partial_t u) \cdot \Lambda^{-2} b \, dx - \int_{\mathbb{T}^n} (\tilde{b} \cdot \nabla u) \cdot \Lambda^{-2} \partial_t b \, dx\notag\\[1ex]
&=: I_4 + I_5.
\end{align}
To estimate $I_4$, we apply \eqref{equation}$_1$ to express it as
\begin{align*}
I_4&=-\int_{\mathbb{T}^n} (\tilde{b} \cdot \nabla (\tilde{b} \cdot \nabla b + \mathbb{P}(b \cdot \nabla b - u \cdot \nabla u)) \cdot \Lambda^{-2} b \, dx\\[1ex]
&=-\int_{\mathbb{T}^n} (\tilde{b} \cdot \nabla) (\tilde{b} \cdot \nabla b) \cdot \Lambda^{-2} b \, dx - \int_{\mathbb{T}^n} (\tilde{b} \cdot \nabla \mathbb{P}(b \cdot \nabla b - u \cdot \nabla u)) \cdot \Lambda^{-2} b \, dx\\[1ex]
&=:I_{41}+I_{42}.
\end{align*}
By integration by parts, \eqref{meanzero2}, and the inequality $\|b\|_{L^2} \le \|\Lambda b\|_{L^2}$, we deduce
\begin{align*}
I_{41}
&= -\int_{\mathbb{T}^n} (\tilde{b} \cdot \nabla)(\tilde{b} \cdot \nabla b) \cdot \Lambda^{-2} b \, dx \\
&= \int_{\mathbb{T}^n} (\tilde{b} \cdot \nabla b) \cdot \Lambda^{-2} (\tilde{b} \cdot \nabla b) \, dx
= \|\Lambda^{-1} (\tilde{b} \cdot \nabla b)\|_{L^2}^2 \\
&\le |\tilde{b}|^2 \, \| \Lambda^{-1} \nabla b \|_{L^2}^2
\le |\tilde{b}|^2 \, \| b \|_{L^2}^2
\le |\tilde{b}|^2 \, \|\Lambda b\|_{L^2}^2.
\end{align*}
Using the boundedness of the Leray projection \( \mathbb{P} \) on \( L^2 \), one has
\begin{align*}
I_{42}
&= - \int_{\mathbb{T}^n} (\tilde{b} \cdot \nabla \mathbb{P}(b \cdot \nabla b - u \cdot \nabla u)) \cdot \Lambda^{-2} b \, dx  \\[1ex]
&=  \int_{\mathbb{T}^n} \mathbb{P}(b \cdot \nabla b - u \cdot \nabla u) \cdot \Lambda^{-2} (\tilde{b}\cdot\nabla  b) \, dx  \\[1ex]
&\leq \lVert \tilde{b} \cdot \nabla \Lambda^{-2} b \rVert_{L^2}\left(\lVert b \cdot \nabla b \rVert_{L^2}+\lVert u \cdot \nabla u \rVert_{L^2}\right)  \\[1ex]
&\leq C(|\tilde{b}|) \lVert b \rVert_{L^2} \left(\lVert \nabla b \rVert_{L^\infty}\lVert b \rVert_{L^2}+\lVert \nabla u \rVert_{L^\infty}\lVert u \rVert_{L^2}  \right)  \\[1ex]
&\leq C(|\tilde{b}|)\left(\|\nabla b\|_{L^\infty}+\|\nabla u\|_{L^\infty}\right) \left(\| b\|_{L^2}^2+\| u\|_{L^2}^2\right).
\end{align*}
Thus,
\begin{align}\label{8.1a}
I_{4}&\leq |\tilde{b}|^2\lVert \Lambda b \rVert_{L^2}^2+ C\left(\left\| \nabla u \right\|_{L^\infty}+\left\| \nabla b \right\|_{L^\infty} \right)  \left( \left\| u \right\|_{L^2}^2 + \left\|  b \right\|_{L^2}^2 \right).
\end{align}
For \(I_{5}\), one can use \eqref{equation}$_2$ and integration by parts to get
\begin{align*}
I_{5} &= -\int_{\mathbb{T}^n} (\tilde{b} \cdot \nabla u) \cdot \Lambda^{-2}\left( \tilde{b} \cdot \nabla u + \Delta b - (u \cdot \nabla b - b \cdot \nabla u) \right)\, dx \\[1ex]
&= -\lVert \Lambda^{-1}(\tilde{b} \cdot \nabla u) \rVert_{L^2}^2 - \int_{\mathbb{T}^n} (\tilde{b} \cdot \nabla u) \cdot \Lambda^{-2} \Delta b\, dx + \int_{\mathbb{T}^n} \Lambda^{-2}(\tilde{b} \cdot \nabla u) \cdot (u \cdot \nabla b - b \cdot \nabla u) \,dx \\[1ex]
&=: I_{51} +I_{52} + I_{53}.
\end{align*}
For $I_{52}$, by H\"{o}lder and Young inequalities, we have
\begin{align*}
|I_{52}| &= \left| \int_{\mathbb{T}^n} (\tilde{b} \cdot \nabla u) \cdot \Lambda^{-2} \Delta b\, dx \right| \\[1ex]
&\leq \lVert \Lambda^{-1}(\tilde{b} \cdot \nabla u) \rVert_{L^2} \lVert \Lambda b \rVert_{L^2} \\[1ex]
&\leq \frac{1}{2} \lVert \Lambda^{-1}(\tilde{b} \cdot \nabla u) \rVert_{L^2}^2 + \frac{1}{2} \lVert \Lambda b \rVert_{L^2}^2.
\end{align*}
For $I_{53}$, we have
\begin{align*}
I_{53}&= \int_{\mathbb{T}^n} \Lambda^{-2} (\tilde{b} \cdot \nabla u) \cdot (u \cdot \nabla b - b \cdot \nabla u) \,dx\\[1ex]
&\leq C \lVert \tilde{b} \cdot \nabla \Lambda^{-2} u\rVert_{L^2}\left(\lVert u \cdot \nabla b \rVert_{L^2}+\lVert b \cdot \nabla u \rVert_{L^2}\right)  \\[1ex]
&\leq C(|\tilde{b}|) \lVert  u\rVert_{L^2} \left(\lVert \nabla b \rVert_{L^\infty}\lVert u \rVert_{L^2} +\lVert \nabla u \rVert_{L^\infty}\lVert b \rVert_{L^2} \right)  \\[1ex]
&\leq C(|\tilde{b}|)\left(\|\nabla b\|_{L^\infty}+\|\nabla u\|_{L^\infty}\right) \left(\| b\|_{L^2}^2+\| u\|_{L^2}^2\right).
\end{align*}
The above estimates yield
\begin{align}\label{8.2a}
I_{5} \leq -\frac{1}{2}\lVert \Lambda^{-1}(\tilde{b} \cdot \nabla u) \rVert_{L^2}^2 +\frac{1}{2}\lVert \Lambda b \rVert_{L^2}^2+ C\left(\lVert \nabla u \rVert_{L^{\infty}}+\lVert \nabla b \rVert_{L^{\infty}}\right)\left(\lVert  u \rVert^2_{L^2} +\lVert b \rVert^2_{L^2}\right).
\end{align}
Finally, substituting \eqref{8.1a} and \eqref{8.2a} into \eqref{J12}, then \eqref{key1b} is obtained. So far, the proof of Proposition \ref{prop3.21} is completed.
\end{proof}

\section{Proof of Theorem \ref{thm}}\label{s:6}
In this section, we establish the nonlinear stability result stated in Theorem \ref{thm}. The analysis is divided into three subsections.

\subsection{A priori estimates}\label{4/1}

To derive the a priori estimates, for the case $\mu=1$, $\nu=0$,  we define the following modified energy functional:
\begin{align}
	Q_s(t) := a \lVert (u, b)(t) \rVert_{H^s}^2
	- \sum_{l = 0}^{s} \int_{\mathbb{T}^n} (\tilde{b} \cdot \nabla b)(t) \cdot \Lambda^{2l-2} u(t) \, dx,
	\quad s\in[0, m],\label{lya}
\end{align}
where \(a > 0\) is a suitably chosen constant. Our focus will be on the evolution of $Q_s(t)$.

We begin by analyzing the first term on the right-hand side of \eqref{lya}.
\begin{prop}\label{prop4.2}
Let $(u,b)$ be a smooth global solution to $\eqref{equation}$ with $\mu =1, \nu = 0$.
Assume $m > 2 + r + \frac{n}{2}$. For any $s \in [0, m]$ and $t \in [0, T]$, it holds that
\begin{align}\label{2.33}
\frac{1}{2}\frac{d}{dt} \lVert( \Lambda^s u, \Lambda^s b)\rVert_{L^2}^2 +\lVert \Lambda^{s + 1} u \rVert_{L^2}^2
\leq C \left( \lVert \nabla u \rVert_{L^{\infty}} + \lVert b \rVert_{H^{m-1-r}}^2 \right) \lVert( \Lambda^s u, \Lambda^s b)\rVert_{L^2}^2.
\end{align}
In particular,
\begin{align}\label{2.33'}
\frac{1}{2}\frac{d}{dt} \lVert (u, b)(t) \rVert_{H^m}^2 +\lVert \nabla u \rVert_{H^m}^2
\leq C \left( \lVert \nabla u \rVert_{L^{\infty}} + \lVert b \rVert_{H^{m-1-r}}^2 \right) \lVert (u, b)(t) \rVert_{H^m}^2.
\end{align}
\end{prop}
\begin{proof}
Applying the operator \(\Lambda^s\) to both sides of $\eqref{equation}_1$ and $\eqref{equation}_2$, taking the inner products of the results with \(\Lambda^su\) and \(\Lambda^sb\) respectively, and then adding them together, we obtain
\[
\frac{1}{2} \frac{d}{dt} \left( \lVert \Lambda^s u \rVert_{L^2}^2 + \lVert \Lambda^s b \rVert_{L^2}^2 \right) + \lVert \Lambda^{s + 1} u \rVert_{L^2}^2  = \sum_{j = 1}^{6} \mathrm{I}_j,
\]
where
\begin{align*}
\mathrm{I}_1 &= - \int_{\mathbb{T}^n} \Lambda^s (u \cdot \nabla u) \cdot \Lambda^s u \, dx, & \mathrm{I}_2 &= \int_{\mathbb{T}^n} \Lambda^s (b \cdot \nabla b) \cdot \Lambda^s u \, dx, \\[1ex]
\mathrm{I}_3 &= - \int_{\mathbb{T}^n} \Lambda^s (u \cdot \nabla b) \cdot \Lambda^s b \, dx, & \mathrm{I}_4 &= \int_{\mathbb{T}^n} \Lambda^s (b \cdot \nabla u) \cdot \Lambda^s b \, dx, \\[1ex]
\mathrm{I}_5 &= \int_{\mathbb{T}^n} (\tilde{b} \cdot \nabla \Lambda^s b) \cdot \Lambda^s u \, dx, & \mathrm{I}_6 &= \int_{\mathbb{T}^n} (\tilde{b} \cdot \nabla \Lambda^s u) \cdot \Lambda^s b \, dx.
\end{align*}
By direct computation and integrating by parts, we find
\[\mathrm{I}_5 + \mathrm{I}_6 = \int_{\mathbb{T}^n} (\tilde{b} \cdot \nabla \Lambda^s b) \cdot \Lambda^s u \, dx + \int_{\mathbb{T}^n} (\tilde{b} \cdot \nabla \Lambda^s u) \cdot \Lambda^s b \, dx = 0.
\]
Making use of Lemma \ref{jiaohuanzi} and \eqref{equation}$_3$, it follows that
\begin{align*}
\mathrm{I}_1 = - \int_{\mathbb{T}^n} \Lambda^s (u \cdot \nabla u) \cdot \Lambda^s u \, dx &= - \int_{\mathbb{T}^n} \left[ \Lambda^s (u \cdot \nabla u) - u \cdot \nabla \Lambda^s u \right] \cdot \Lambda^s u \, dx \\[1ex]
&\leq C \lVert \Lambda^s u \rVert_{L^2}^2 \lVert \nabla u \rVert_{L^{\infty}},
\end{align*}
\begin{align*}
\mathrm{I}_3 = - \int_{\mathbb{T}^n} \Lambda^s (u \cdot \nabla b) \cdot \Lambda^s b \, dx &= - \int_{\mathbb{T}^n} \left[ \Lambda^s (u \cdot \nabla b) - u \cdot \nabla \Lambda^s b \right] \cdot \Lambda^s b \, dx \\[1ex]
&\leq C \lVert \Lambda^s b \rVert_{L^2} \left( \lVert \Lambda^s u \rVert_{L^2} \lVert \nabla b \rVert_{L^{\infty}} + \lVert \Lambda^s b \rVert_{L^2} \lVert \nabla u \rVert_{L^{\infty}} \right),
\end{align*}
and
\begin{align*}
\mathrm{I}_2 + \mathrm{I}_4 &= \int_{\mathbb{T}^n} \Lambda^s (b \cdot \nabla b) \cdot \Lambda^s u \, dx + \int_{\mathbb{T}^n} \Lambda^s (b \cdot \nabla u) \cdot \Lambda^s b \, dx \\
&= \int_{\mathbb{T}^n} \left[ \Lambda^s (b \cdot \nabla b) - b \cdot \Lambda^s \nabla b \right] \cdot \Lambda^s u \, dx \\
&\quad + \int_{\mathbb{T}^n} \left[ \Lambda^s (b \cdot \nabla u) - b \cdot \Lambda^s \nabla u \right] \cdot \Lambda^s b \, dx \\[1ex]
&\leq C \lVert \Lambda^s u \rVert_{L^2} \lVert \Lambda^s b \rVert_{L^2} \lVert \nabla b \rVert_{L^{\infty}} \\[1ex]
&\quad + C \lVert \Lambda^s b \rVert_{L^2} \left( \lVert \Lambda^s u \rVert_{L^2} \lVert \nabla b \rVert_{L^{\infty}} + \lVert \Lambda^s b \rVert_{L^2} \lVert \nabla u \rVert_{L^{\infty}} \right).
\end{align*}
Applying the Sobolev embedding
\begin{align}\label{2.55}
\lVert \nabla b \rVert_{L^{\infty}} \leq C \lVert b \rVert_{H^{m-1-r}}, \quad \text{with } m > 2 + r + \frac{n}{2},
\end{align}
and combining the preceding estimates, we obtain
\begin{align*}
&\frac{1}{2} \frac{d}{dt} \left( \lVert \Lambda^s u \rVert_{L^2}^2 + \lVert \Lambda^s b \rVert_{L^2}^2 \right) + \lVert \Lambda^{s + 1} u \rVert_{L^2}^2\\[1ex]
 \leq &C \lVert \nabla u \rVert_{L^{\infty}} \left( \lVert \Lambda^s u \rVert_{L^2}^2 + \lVert \Lambda^s b \rVert_{L^2}^2 \right)+C\lVert \nabla b \rVert_{L^{\infty}}\lVert \Lambda^s u \rVert_{L^2} \lVert \Lambda^s b \rVert_{L^2} \\[1ex]
 \leq &C \lVert \nabla u \rVert_{L^{\infty}} \left( \lVert \Lambda^s u \rVert_{L^2}^2 + \lVert \Lambda^s b \rVert_{L^2}^2 \right)+C\lVert b \rVert_{H^{m-1-r}}\lVert \Lambda^s u \rVert_{L^2} \lVert \Lambda^s b \rVert_{L^2},\quad (m>2+r+\frac{n}{2}) \\[1ex]
 \leq &C \lVert \nabla u \rVert_{L^{\infty}} \left( \lVert \Lambda^s u \rVert_{L^2}^2 + \lVert \Lambda^s b \rVert_{L^2}^2 \right)+C\lVert b \rVert_{H^{m-1-r}}\lVert \Lambda^{s+1} u \rVert_{L^2} \lVert \Lambda^s b \rVert_{L^2},
\end{align*}
where, in the last step, we used \eqref{meanzero2}, Plancherel's theorem  together with the Fourier multiplier property \(|k|^s \leq |k|^{s + 1}\) for all \(|k| \geq 1\), which implies \(\lVert \Lambda^s u \rVert_{L^2} \leq \lVert \Lambda^{s+1} u \rVert_{L^2}\).

Finally, applying the Young inequality with an arbitrary constant \(\varepsilon > 0\), we conclude
\begin{align*}
&\frac{1}{2} \frac{d}{dt} \left( \lVert \Lambda^s u \rVert_{L^2}^2 + \lVert \Lambda^s b \rVert_{L^2}^2 \right) + \lVert \Lambda^{s + 1} u \rVert_{L^2}^2\\[1ex]
 \leq &C \lVert \nabla u \rVert_{L^{\infty}} \left( \lVert \Lambda^s u \rVert_{L^2}^2 + \lVert \Lambda^s b \rVert_{L^2}^2 \right)+\varepsilon\lVert \Lambda^{s+1} u \rVert^2_{L^2}+C_\varepsilon\lVert b \rVert^2_{H^{m-1-r}} \lVert \Lambda^s b \rVert^2_{L^2}.
\end{align*}
Absorbing the \( \varepsilon \)-term into the left-hand side yields the desired estimate. This completes the proof of Proposition \ref{prop4.2}.
\end{proof}

The norm $\lVert b \rVert_{H^{m-1-r}}$, which appears in Proposition \ref{prop4.2}, plays a key role in controlling the nonlinear terms.
To exploit how dissipation acts on \( b \), we now turn to the second term on the right-hand side of \eqref{lya}.
\begin{prop}\label{prop4.3}
Let $m > 2 + r + \frac{n}{2}$, and let $(u,b)$ be a smooth global solution to $\eqref{equation}$ with $\mu = 1, \nu = 0$.
Then, for any $s \in [0, m]$ and $t \in [0, T]$, there is
\begin{align}\label{2.333}
-\frac{d}{dt} \sum_{l = 0}^{s}\int_{\mathbb{T}^n} (\tilde{b} \cdot \nabla b) \cdot \Lambda^{2l-2} u \, dx
\leq &\left( 1+ |\tilde{b}|^2 \right) \left\lVert \nabla u \right\rVert_{H^s}^2
-\frac{1}{2} \left\lVert \Lambda^{-1} (\tilde{b} \cdot \nabla b) \right\rVert_{H^s}^2 \notag\\[1ex]
&+ C \left\lVert (u, b) \right\rVert_{H^s}^2 \left( \left\lVert \nabla u \right\rVert_{L^{\infty}} + \lVert b \rVert^2_{H^{m-1-r}} \right),
\end{align}
where $C > 0$ depends only on the fixed parameters.
\end{prop}
\begin{proof}
By straightforward calculations, we have
\begin{align*}
-\frac{d}{dt} \int_{\mathbb{T}^n} (\tilde{b} \cdot \nabla b) \cdot \Lambda^{2l-2} u \, dx &= -\int_{\mathbb{T}^n} (\tilde{b} \cdot \nabla \partial_t b) \cdot \Lambda^{2l-2} u \, dx - \int_{\mathbb{T}^n} (\tilde{b} \cdot \nabla b) \cdot \Lambda^{2l-2} \partial_t u \, dx\\[1ex]
&=: \mathrm{II}_1 + \mathrm{II}_2.
\end{align*}
To proceed, we first focus on estimating \( \mathrm{II}_1 \). Due to its complexity, \( \mathrm{II}_1 \) is further divided into two parts, \( \mathrm{II}_{11} \) and \( \mathrm{II}_{12} \). After completing these estimates, we then turn our attention to \( \mathrm{II}_2 \).

\smallskip
\noindent
\textbf{Estimate of \(\mathrm{II}_1\):} Using $\eqref{equation}_2$, we rewrite it as
\begin{align*}
\mathrm{II}_1&=-\int_{\mathbb{T}^n} (\tilde{b} \cdot \nabla (\tilde{b} \cdot \nabla u + b \cdot \nabla u - u \cdot \nabla b)) \cdot \Lambda^{2l-2} u \, dx\\[1ex]
&=-\int_{\mathbb{T}^n} (\tilde{b} \cdot \nabla) (\tilde{b} \cdot \nabla u) \cdot \Lambda^{2l-2} u \, dx - \int_{\mathbb{T}^n} \tilde{b} \cdot \nabla (b \cdot \nabla u - u \cdot \nabla b) \cdot \Lambda^{2l-2} u \, dx\\[1ex]
&=:\mathrm{II}_{11}+\mathrm{II}_{12}.
\end{align*}
For $\mathrm{II}_{11}$, by integration by parts and using the commutation property
\[
(\tilde{b} \cdot \nabla)\Lambda^{2l-2} = \Lambda^{2l-2}(\tilde{b} \cdot \nabla) \quad \text{for all } l \in \mathbb{R},
\]
which holds since $\tilde{b}$ is a constant vector field, together with Plancherel's theorem, we obtain
\begin{align*}
\mathrm{II}_{11} &= -\int_{\mathbb{T}^n} (\tilde{b} \cdot \nabla) (\tilde{b} \cdot \nabla u) \cdot \Lambda^{2l-2} u \, dx = \int_{\mathbb{T}^n} (\tilde{b} \cdot \nabla u) \cdot (\tilde{b} \cdot \nabla \Lambda^{2l-2} u) \, dx\\[1ex]
&= \int_{\mathbb{T}^n} (\tilde{b} \cdot \nabla u) \cdot \Lambda^{2l-2}(\tilde{b} \cdot \nabla u) \, dx = \lVert \Lambda^{l-1}(\tilde{b} \cdot \nabla u) \rVert_{L^2}^2\\[1ex]
& \leq |\tilde{b}|^2 \lVert \Lambda^{l} u \rVert_{L^2}^2 \leq |\tilde{b}|^2 \lVert \Lambda^{l+1} u \rVert_{L^2}^2.
\end{align*}
To estimate \( \mathrm{II}_{12} \), we distinguish the case \( l \geq 1 \) and the case \( l = 0 \).

\smallskip
\noindent
{\textbf{Case $l \geq 1$.}} We employ integration by parts, H\"{o}lder inequality, Lemma \ref{jiaohuanzi}, the Sobolev embedding \eqref{2.55} and the Young inequality to control high-order derivatives.
\begin{align*}
\mathrm{II}_{12}
&= - \int_{\mathbb{T}^n} (\tilde{b} \cdot \nabla)(b \cdot \nabla u - u \cdot \nabla b) \cdot \Lambda^{2l-2} u \, dx \\[1ex]
&= - \int_{\mathbb{T}^n} (\tilde{b} \cdot \nabla)(b \cdot \nabla u - u \cdot \nabla b) \cdot \Lambda^{l-2} (\Lambda^{l} u) \, dx \\[1ex]
&\leq C \| \tilde{b} \cdot \nabla \Lambda^{l-2}(b \cdot \nabla u - u \cdot \nabla b) \|_{L^2} \| \Lambda^{l} u \|_{L^2} \\[1ex]
&\leq C(|\tilde{b}|) \big( \| \Lambda^{l-1} b \|_{L^2} \| \nabla u \|_{L^\infty} + \| b \|_{L^\infty} \| \Lambda^{l} u \|_{L^2} \big) \| \Lambda^{l} u \|_{L^2} \\[1ex]
&\quad + C(|\tilde{b}|) \big( \| \Lambda^{l-1} u \|_{L^2} \| \nabla b \|_{L^\infty} + \| u \|_{L^\infty} \| \Lambda^{l} b \|_{L^2} \big) \| \Lambda^{l} u \|_{L^2} \\[1ex]
&\leq C(|\tilde{b}|) \| \nabla u \|_{L^\infty} \big( \| \Lambda^{l} u \|_{L^2}^2 + \| \Lambda^{l} b \|_{L^2}^2 \big)
+ \frac{1}{4} \| \Lambda^{l+1} u \|_{L^2}^2 + C \| b \|_{H^{m - 1 - r}}^2 \| \Lambda^{l} u \|_{L^2}^2.
\end{align*}

\smallskip
\noindent
{\textbf{Case $l = 0$.}} In this case, the term $\Lambda^{-2}u$ involves a negative-order Sobolev norm, making Lemma \ref{jiaohuanzi} inapplicable. To circumvent the associated singularity, we instead exploit the $L^2$-boundedness of the Riesz transforms $R_i = \partial_i \Lambda^{-1}$ together with direct $L^2$ estimates.
\begin{align*}
\mathrm{II}_{12}
&= - \int_{\mathbb{T}^n} (\tilde{b} \cdot \nabla)(b \cdot \nabla u-u \cdot \nabla b) \cdot \Lambda^{-2} u \, dx \\[1ex]
&= - \int_{\mathbb{T}^n} \tilde{b}_i \partial_i \Lambda^{-1}(b_j \partial_j u_k-u_j \partial_j b_k) \cdot \Lambda^{-1} u_k \, dx \\[1ex]
&= \int_{\mathbb{T}^n} \tilde{b}_i R_i(b_j \partial_j u_k-u_j \partial_j b_k) \cdot \Lambda^{-1} u_k \, dx \\[1ex]
&\leq C \| \tilde{b} \|_{L^\infty}\left(\| R_i(b_j \partial_j u_k) \|_{L^2}+\| R_i (u_j \partial_j b_k) \|_{L^2}\right) \| \Lambda^{-1} u_k \|_{L^2} \\[1ex]
&\leq C(|\tilde{b}|) \| b \cdot \nabla u \|_{L^2} \| \Lambda^{-1} u \|_{L^2}+C(|\tilde{b}|)\|u \cdot \nabla b\|_{L^2}\| \Lambda^{-1} u\|_{L^2}\\[1ex]
&\leq C(|\tilde{b}|) \| \nabla u \|_{L^\infty} \| b \|_{L^2} \| u \|_{L^2}+C(|\tilde{b}|)\|\nabla b\|_{L^\infty}\| u\|_{L^2}\|\nabla u\|_{L^2} \\[1ex]
&\leq C(|\tilde{b}|) \left\| \nabla u \right\|_{L^\infty} \left( \left\| u \right\|_{L^2}^2 + \left\| b \right\|_{L^2}^2 \right)+ \frac{1}{4}\left\| \nabla u \right\|_{L^2}^2 +  C\left\| b \right\|_{H^{m - 1 - r}}^2 \left\| u \right\|_{L^2}^2.
\end{align*}
Thus, for any $l\geq0$
\begin{align}\label{8.1}
|\mathrm{II}_1|&\leq \left(|\tilde{b}|^2+\frac{1}{4}\right) \lVert \Lambda^{l+1} u \rVert_{L^2}^2+ C\left(\left\| \nabla u \right\|_{L^\infty}+\left\| b \right\|_{H^{m - 1 - r}}^2 \right)  \left(\| \Lambda^{l} u \|_{L^2}^2 +\| \Lambda^{l} b \|_{L^2}^2 \right).
\end{align}

We now turn  our attention to \(\mathrm{II}_2\).

\smallskip
\noindent
\textbf{Estimate of \(\mathrm{II}_2\):} Similarly, by using the first equation in \eqref{1.41} and integration by parts, we get
\begin{align*}
\mathrm{II}_2 &= -\int_{\mathbb{T}^n} (\tilde{b} \cdot \nabla b) \cdot \Lambda^{2l-2}\left( \tilde{b} \cdot \nabla b + \Delta u - \mathbb{P}(u \cdot \nabla u - b \cdot \nabla b) \right) \,dx \\[1ex]
&= -\lVert \Lambda^{l-1}(\tilde{b} \cdot \nabla b) \rVert_{L^2}^2 - \int_{\mathbb{T}^n} (\tilde{b} \cdot \nabla b) \cdot \Lambda^{2l-2} \Delta u \,dx + \int_{\mathbb{T}^n} (\tilde{b} \cdot \nabla b) \cdot \Lambda^{2l-2} \mathbb{P}(u \cdot \nabla u - b \cdot \nabla b) \,dx \\[1ex]
&=: \mathrm{II}_{21} + \mathrm{II}_{22} + \mathrm{II}_{23}.
\end{align*}
For $\mathrm{II}_{22}$, by H\"{o}lder and Young inequalities with a parameter $\varepsilon>0$, we obtain
\begin{align*}
|\mathrm{II}_{22}|
&= \left| \int_{\mathbb{T}^n} (\tilde{b} \cdot \nabla b) \cdot \Lambda^{2l-2} \Delta u\, dx \right| \\[1ex]
&\leq \lVert \Lambda^{l-1}(\tilde{b} \cdot \nabla b) \rVert_{L^2} \,
      \lVert \Lambda^{l+1} u \rVert_{L^2} \\[1ex]
&\leq \varepsilon \lVert \Lambda^{l+1} u \rVert_{L^2}^2
   + \frac{1}{4\varepsilon} \lVert \Lambda^{l-1}(\tilde{b} \cdot \nabla b) \rVert_{L^2}^2.
\end{align*}
Choosing $\varepsilon =\frac{3}{4}$ yields
\begin{align*}
|\mathrm{II}_{22}|
&\leq \frac{3}{4}\lVert \Lambda^{l+1} u \rVert_{L^2}^2
   + \frac{1}{3} \lVert \Lambda^{l-1}(\tilde{b} \cdot \nabla b) \rVert_{L^2}^2.
\end{align*}

For $\mathrm{II}_{23}$, similar to $\mathrm{II}_{12}$, we distinguish between the cases $l \geq 1$  and $l = 0$.

\smallskip
\noindent
{\textbf{Case $l \geq 1$.}}
Using the boundedness of the Leray projection \( \mathbb{P} \) on \( L^2 \), Lemma \ref{jiaohuanzi}, the Sobolev embedding \eqref{2.55}, Poincar\'e and Young inequalities, we obtain
\begin{align*}
|\mathrm{II}_{23}| &\leq \lVert \Lambda^{l-1}(\tilde{b} \cdot \nabla b) \rVert_{L^2} \lVert \Lambda^{l-1}(u \cdot \nabla u - b \cdot \nabla b) \rVert_{L^2} \\[1ex]
&\leq C \lVert \Lambda^{l-1}(\tilde{b} \cdot \nabla b) \rVert_{L^2} \left( \lVert \Lambda^{l} u \rVert_{L^2} \lVert \nabla u \rVert_{L^\infty} + \lVert \Lambda^{l} b \rVert_{L^2} \lVert b \rVert_{H^{m - 1 - r}} \right)\\[1ex]
&\leq C(|\tilde{b}|)\lVert \Lambda^{l} b \rVert_{L^2} \lVert \Lambda^{l} u \rVert_{L^2}\lVert \nabla u \rVert_{L^{\infty}}+C\lVert \Lambda^{l-1} (\tilde{b} \cdot \nabla b) \rVert_{L^2}\lVert \Lambda^{l} b \rVert_{L^2}\lVert b \rVert_{H^{m-1-r}}\\[1ex]
&\leq C(|\tilde{b}|)\lVert \nabla u \rVert_{L^{\infty}}\left(\lVert \Lambda^{l} u \rVert^2_{L^2} +\lVert \Lambda^{l} b \rVert^2_{L^2}\right) +\frac{1}{6}  \lVert \Lambda^{l-1} (\tilde{b} \cdot \nabla b) \rVert_{L^2}^2 +  C\left\| b \right\|_{H^{m - 1 - r}}^2 \|\Lambda^{l} b \|_{L^2}^2.
\end{align*}

\smallskip
\noindent
{\textbf{Case $l =0$.}}
In this case, we proceed via direct \( L^2 \) estimates
\begin{align*}
\mathrm{II}_{23}&= \int_{\mathbb{T}^n} (\tilde{b} \cdot \nabla b) \cdot \Lambda^{-2} \mathbb{P}(u \cdot \nabla u - b \cdot \nabla b)\, dx\\[1ex]
&\leq \lVert \Lambda^{-2}(\tilde{b} \cdot \nabla b) \rVert_{L^2}\left(\lVert u \cdot \nabla u \rVert_{L^2}+\lVert b \cdot \nabla b \rVert_{L^2}\right)  \\[1ex]
&\leq C \lVert \Lambda^{-1}(\tilde{b} \cdot \nabla b) \rVert_{L^2} \left(\lVert \nabla u \rVert_{L^\infty}\lVert u \rVert_{L^2}+\lVert \nabla b \rVert_{L^\infty}\lVert b \rVert_{L^2}  \right)  \\[1ex]
&\leq C(|\tilde{b}|) \lVert b \rVert_{L^2} \lVert u \rVert_{L^2}\lVert  \nabla u \rVert_{L^\infty}+C \lVert \Lambda^{-1}(\tilde{b} \cdot \nabla b)\rVert_{L^2} \|b\|_{H^{m-1-r}} \lVert b \rVert_{L^2}\\[1ex]
&\leq C(|\tilde{b}|)\lVert \nabla u \rVert_{L^{\infty}}\left(\lVert  u \rVert^2_{L^2} +\lVert b \rVert^2_{L^2}\right) +\frac{1}{6}\lVert  \Lambda^{-1} (\tilde{b} \cdot \nabla b) \rVert_{L^2}^2 +  C\left\| b \right\|_{H^{m - 1 - r}}^2 \left\| b  \right\|_{L^2}^2.
\end{align*}
Combining the estimates for $\mathrm{II}_{21}, \mathrm{II}_{22}$ and $\mathrm{II}_{23}$, we conclude that for any $l\geq0$,
\begin{align}\label{8.2}
|\mathrm{II}_2|&\leq -\frac{1}{2} \lVert \Lambda^{l-1}(\tilde{b} \cdot \nabla b) \rVert_{L^2}^2 +\frac{3}{4} \lVert \Lambda^{l+1} u \rVert_{L^2}^2\notag\\[1ex]
&\quad+ C\left(\lVert \nabla u \rVert_{L^{\infty}}+\left\| b \right\|_{H^{m - 1 - r}}^2\right)\left(\lVert \Lambda^{l} u \rVert^2_{L^2} +\lVert \Lambda^{l} b \rVert^2_{L^2}\right).
\end{align}
Thus, combining \eqref{8.1} and \eqref{8.2}, and using the assumption \( m > 2 + r + \frac{n}{2} \), we obtain
\begin{align*}
- \frac{d}{dt} \sum_{l = 0}^{s} \int_{\mathbb{T}^n} (\tilde{b} \cdot \nabla b) \cdot \Lambda^{2l-2} u \, dx
\leq{}& \left( 1+ |\tilde{b}|^2 \right) \left\lVert \nabla u \right\rVert_{H^s}^2
- \frac{1}{2} \left\lVert \Lambda^{-1} (\tilde{b} \cdot \nabla b) \right\rVert_{H^s}^2 \\[1ex]
& + C \left\lVert (u, b) \right\rVert_{H^s}^2 \left( \left\lVert \nabla u \right\rVert_{L^{\infty}} + \left\lVert b \right\rVert_{H^{m-1-r}}^2 \right).
\end{align*}
This completes the proof of Proposition \ref{prop4.3}.
\end{proof}

With the help of Propositions \ref{prop4.2} and \ref{prop4.3}, we proceed to derive the energy estimates.
Define the energy functional
\begin{align*}
E_m^2(T) :=  \sup_{t \in [0, T]} \lVert (u, b)(t) \rVert_{H^m}^2
+ \int_0^T \lVert \nabla u(t) \rVert_{H^m}^2 \, dt
+ \int_0^T \lVert \Lambda^{-1}(\tilde{b} \cdot \nabla b)(t) \rVert_{H^m}^2 \, dt.
\end{align*}
Then, the following proposition holds.
\begin{prop}\label{prop3.211a}
Let \( n \in \mathbb{N} \) with \( n \geq 2 \), \(r>n-1\), and \( m \in \mathbb{N} \) satisfying \( m > 2 + r + \frac{n}{2} \).
Assume that \( (u, b) \) is a smooth global solution to \eqref{equation} with $\mu=1$, $\nu=0$.
Then there exists a constant \( C > 0 \) such that
\begin{align}\label{12.24}
E_m^2(T) \leq C(|\tilde{b}|) \lVert (u_0, b_0) \rVert_{H^m}^2
+ C E_m^4(T)
+ C E_m^2(T) \int_0^T \lVert \nabla u(t) \rVert_{L^{\infty}} \, dt,
\end{align}
for all \( T > 0 \).
\end{prop}
\begin{proof}
Recall the modified energy functional defined in \eqref{lya}:
\begin{align*}
Q_m(t):= a\lVert (u, b)(t) \rVert_{H^m}^2 - \sum_{l = 0}^{m} \int_{\mathbb{T}^n} (\tilde{b} \cdot \nabla b)(t) \cdot \Lambda^{2l-2} u(t) \, dx.
\end{align*}
Multiplying \eqref{2.33'} by $a := 1+\tfrac{|\tilde{b}|}{2}+\tfrac{|\tilde{b}|^2}{2}$
and adding $\frac{1}{2}\eqref{2.333}$, we obtain
\begin{align}\label{12.2}
\frac{1}{2} \frac{d}{dt} Q_m(t)
\leq\; & -a \left\lVert \nabla u \right\rVert_{H^m}^2 - \frac{1}{4} \left\lVert \Lambda^{-1} (\tilde{b} \cdot \nabla b) \right\rVert_{H^m}^2+ \left( \frac{1}{2} + \frac{|\tilde{b}|^2}{2} \right) \left\lVert \nabla u \right\rVert_{H^m}^2 \notag\\
&  + C \lVert (u, b) \rVert_{H^m}^2 \left( \left\lVert \nabla u \right\rVert_{L^{\infty}} + \lVert b \rVert_{H^{m-1-r}}^2 \right) \notag\\
\leq\; & -\frac{1}{4} \left( \left\lVert \nabla u \right\rVert_{H^m}^2 + \left\lVert \Lambda^{-1} (\tilde{b} \cdot \nabla b) \right\rVert_{H^m}^2 \right) + C \lVert (u, b) \rVert_{H^m}^2 \left( \left\lVert \nabla u \right\rVert_{L^{\infty}} + \lVert b \rVert_{H^{m-1-r}}^2 \right).
\end{align}
To further control the nonlinear terms on the right-hand side, we apply Lemma \ref{2.1} together with the Poincar\'e inequality to deduce
\begin{align*}
\|b\|_{H^{m-1-r}}
&\leq c \|\tilde{b} \cdot \nabla b\|_{H^{m-1}}
\leq C \|\tilde{b} \cdot \nabla b\|_{\dot{H}^{m-1}}\notag\\
& = C \left\|\Lambda^{m-1}(\tilde{b} \cdot \nabla b)\right\|_{L^2}= C \left\|\Lambda^{-1}(\tilde{b} \cdot \nabla b)\right\|_{\dot{H}^m} \notag\\
&\leq C \left\lVert \Lambda^{-1} (\tilde{b} \cdot \nabla b) \right\rVert_{H^m},
\end{align*}
so that inequality \eqref{12.2} becomes
\begin{align}\label{12.233'}
&\frac{d}{dt} Q_m(t) + \frac{1}{2} \left( \left\lVert \nabla u \right\rVert_{H^m}^2 + \left\lVert \Lambda^{-1} (\tilde{b} \cdot \nabla b) \right\rVert_{H^m}^2 \right)\notag\\
&\qquad\leq C \lVert (u, b) \rVert_{H^m}^2 \left( \left\lVert \nabla u \right\rVert_{L^{\infty}} + \left\lVert \Lambda^{-1} (\tilde{b} \cdot \nabla b) \right\rVert_{H^m}^2 \right).
\end{align}
Using H\"older, Poincar\'e and Young inequalities, one has
\begin{align*}
\left| \int_{\mathbb{T}^n} (\tilde{b} \cdot \nabla b) \cdot \Lambda^{2l-2} u \, dx \right|
\leq |\tilde{b}| \cdot \lVert \Lambda^l b \rVert_{L^2} \cdot \lVert \Lambda^{l - 1} u \rVert_{L^2} \leq \frac{|\tilde{b}|}{2} \left( \lVert \Lambda^l b \rVert_{L^2}^2 + \lVert \Lambda^l u \rVert_{L^2}^2 \right),
\end{align*}
from which we obtain the bound
\begin{align}\label{12.21}
\left| \sum_{l = 0}^{m} \int_{\mathbb{T}^n} (\tilde{b} \cdot \nabla b) \cdot \Lambda^{2l-2} u \, dx \right|
\leq \frac{|\tilde{b}|}{2} \lVert (u, b) \rVert_{H^m}^2.
\end{align}
Integrating \eqref{12.233'} over \([0, T]\), and using the estimate \eqref{12.21}, we arrive at \eqref{12.24}.
\end{proof}

\subsection{Global-in-time existence}\label{4/2}
In this subsection, we establish the global existence part of Theorem \ref{thm}.
The main difficulty lies in controlling the critical term
\[
\int_0^T \|\nabla u(t)\|_{L^\infty} \, dt,
\]
which appears in the energy inequality \eqref{12.24}. To handle this, we employ spectral analysis to derive suitable upper bounds.
The first step is given by the following proposition.
\begin{prop}\label{prop5.1}
Let \( n \in \mathbb{N} \) with \( n \ge 2 \), and choose \( m \in \mathbb{N} \) so that \( m > \frac{n}{2} \). Assume that $(u, b)$ is a smooth global solution to \eqref{equation} with \( \mu = 1 \) and \( \nu = 0 \).
Then there exists a constant $C > 0$ such that for all $T>0$,
\begin{align}\label{prop1.1}
 \sum_{k \in \mathbb{Z}^n \setminus \{0\}} \int_{0}^{T} |k|\left| \hat{u}(t,k) \right| \, dt
 &\leq C\lVert u_0 \rVert_{H^m}
 + C\sup_{t \in [0, T]} \lVert u \rVert_{H^m}\sum_{k \in \mathbb{Z}^n \setminus \{0\}}\int_{0}^{T} |k| |\hat{u}(t,k)| \, dt\notag\\[1ex]
 &\quad +\left(C\sup_{t \in [0, T]} \lVert b \rVert_{H^m}+C(|\tilde{b}|)\right)\sum_{k \in \mathbb{Z}^n \setminus \{0\}} \int_{0}^{T}| \hat{b}(t,k)| \, dt,
\end{align}
where $C(|\tilde{b}|)$ denotes a constant depending on $|\tilde{b}|$.
\end{prop}
\begin{proof}
Applying Duhamel's principle to $\eqref{97}_1$, we obtain
\begin{align}\label{le'}
&\sum_{k \in \mathbb{Z}^n \setminus \{0\}} \int_{0}^{T} |k| |\hat{u}(t,k)| \, dt \leq J_1 + J_2 + J_3 + J_4,
\end{align}
where
\begin{align*}
&J_1:= \sum_{k \in \mathbb{Z}^n \setminus \{0\}} \int_{0}^{T} |k| e^{-|k|^2 t} |\hat{u}_0(k)| \, dt,\notag\\[1ex]
&J_2:= \sum_{k \in \mathbb{Z}^n \setminus \{0\}} \int_{0}^{T} \int_{0}^{t} |k| e^{-|k|^2 (t - \tau)} |\widehat{u \cdot \nabla u}(\tau,k)| \, d\tau dt,\notag\\[1ex]
&J_3 := \sum_{k \in \mathbb{Z}^n \setminus \{0\}} \int_{0}^{T} \int_{0}^{t} |k| e^{-|k|^2 (t - \tau)} |\widehat{b \cdot \nabla b} (\tau,k)| \, d\tau dt,\notag\\[1ex]
&J_4:= \sum_{k \in \mathbb{Z}^n \setminus \{0\}} \int_{0}^{T} \int_{0}^{t} |k| e^{-|k|^2 (t - \tau)} |(\tilde{b} \cdot k) \hat{b} (\tau,k)| \, d\tau dt.
\end{align*}
Here we have used the property that $|\widehat{\mathbb{P} f}(k)| \leq |\hat{f}(k)|$ when estimating $J_2$ and $J_3$.

It is clear that
$$
J_1\leq \sum_{k \in \mathbb{Z}^n \setminus \{0\}} \int_0^T |k|^2 e^{-|k|^2 t} |\hat{u}_0| \, dt \leq \sum_{k \in \mathbb{Z}^n \setminus \{0\}} |\hat{u}_0(k)| \leq C \left\lVert u_0 \right\rVert_{\dot{H}^m}.
$$

For $J_4$, we have
\begin{align*}
J_4 \leq C(|\tilde{b}|) \sum_{k \in \mathbb{Z}^n \setminus \{0\}} \int_{0}^{T} | \hat{b}(t,k) | \, dt.
\end{align*}

For \(J_2\) and \(J_3\), we can deduce
 \begin{align*}
J_2+J_3 &\leq C \sum_{k \in \mathbb{Z}^n \setminus \{0\}} \int_{0}^{T} \left|(\hat{u}*\hat{u}) (t,k)\right| \, d t+C \sum_{k \in \mathbb{Z}^n \setminus \{0\}} \int_{0}^{T}|(\hat{b}*\hat{b}) (t,k)| \, d t
\\[1ex]
&\leq C \sup_{t \in [0, T]} \| u(t) \|_{H^m} \left( \sum_{k \in \mathbb{Z}^n \setminus \{0\}} \int_{0}^{T} |k| |\hat{u} (t,k)| \, dt \right)\\[1ex]
&\quad+C \sup_{t \in [0, T]} \left\lVert b  \right\rVert_{H^m} \left( \sum_{k \in \mathbb{Z}^n \setminus \{0\}} \int_{0}^{T} |\hat{b} (t,k)| \, dt \right).
\end{align*}
 Collecting all the estimates \(J_1\), \(J_2\), \(J_3\) and \(J_4\) leads to \eqref{le'}, which completes the proof of Proposition \ref{prop5.1}.
\end{proof}

In order to close the estimate in Proposition \ref{prop5.1}, it remains to control the term involving
\[\sum_{k \in \mathbb{Z}^n \setminus \{0\}}\int_0^T |\hat{b}(t,k)|\,dt.\]
The following proposition provides the corresponding bound.
\begin{prop}\label{prop5.2}
Let $n \in \mathbb{N}$ with $n \geq 2$ and $m \in \mathbb{N}$ satisfying $m > 3 + 2r + \frac{n}{2}$. Assume that $(u, b)$ is a smooth global solution to \eqref{equation} with \( \mu = 1 \) and \( \nu = 0 \). Then there exists a constant $C > 0$ such that
\begin{align}\label{prop1.2}
 \sum_{k \in \mathbb{Z}^n \setminus \{0\}} \int_{0}^{T}| \hat{b}(t,k)| \, dt&\leq  C\lVert(u_0, b_0) \rVert_{H^m}+C\sup_{t \in [0, T]} \lVert (u, b) \rVert_{H^m}\sum_{k \in \mathbb{Z}^n \setminus \{0\}}\int_{0}^{T} |k| |\hat{u}(t,k)| \, dt\notag\\[1ex]
 &\quad+C\sup_{t \in [0, T]} \lVert (u, b) \rVert_{H^m}\sum_{k \in \mathbb{Z}^n \setminus \{0\}}\int_{0}^{T} |\hat{b}(t,k)| \, dt,
\end{align}
for all $T>0$.
\end{prop}
\begin{proof}
From $\eqref{2.4}_2$, we obtain
\begin{align}
&\sum_{k \in \mathbb{Z}^n \setminus \{0\}} \int_{0}^{T} |\hat{b}(t,k)| \, dt\notag\\[1ex]
\leq& \sum_{k \in \mathbb{Z}^n \setminus \{0\}} \int_{0}^{T} |\widehat{K_1}||\hat{\boldsymbol{\psi}}_0|\, dt+\sum_{k \in \mathbb{Z}^n \setminus \{0\}} \int_{0}^{T} \int_{0}^{t}
|\widehat{K_1}(t-\tau)||\widehat{N}(\tau)| \, d\tau dt\notag\\[1ex]
&+\sum_{k \in \mathbb{Z}^n \setminus \{0\}} \int_{0}^{T} |\widehat{K_3}||\hat{\boldsymbol{\psi}}_0|  \, dt+\sum_{k \in \mathbb{Z}^n \setminus \{0\}} \int_{0}^{T} \int_{0}^{t}
|\widehat{K_3}(t-\tau)||\widehat{N_2}(\tau)| \, d\tau dt\notag\\[1ex]
=:&J_5 + J_6 + J_7 + J_8.\label{le''}
\end{align}
We first estimate \(J_7\) and \(J_8\). By Proposition \ref{pro2} and \eqref{2.3}, we have
\begin{align*}
J_7 \leq \sum_{k \in \mathbb{Z}^n \setminus \{0\}} \int_0^T e^{-\frac{|k|^2}{2} t} |\hat{\boldsymbol{\psi}}_0(k)| \, dt \leq C \| \boldsymbol{\psi}_0 \|_{H^m}.
\end{align*}
For \(J_8\), we proceed similarly to the estimate for \(J_2\) and use Proposition \ref{pro2} to get
\begin{align*}
J_8
&\leq C \int_0^T \sum_{k \in \mathbb{Z}^n \setminus \{0\}}
|(\hat{b} * \hat{u})(t,k)| \, dt \\[1ex]
&\leq C \int_0^T \left( \sum_{k \in \mathbb{Z}^n \setminus \{0\}} |\hat{b}(t,k)| \right)
\left( \sum_{k \in \mathbb{Z}^n \setminus \{0\}} |\hat{u}(t,k)| \right) dt \\[1ex]
&\leq C \sup_{t \in [0, T]} \| u(t) \|_{H^m} \sum_{k \in \mathbb{Z}^n \setminus \{0\}} \int_0^T |\hat{b}(t,k)| \, dt.
\end{align*}
We now estimate \(J_5\) and \(J_6\).  As the previous estimates, we have
\begin{align*}
J_5
&\leq \sum_{k \in S_1 \cup S_2} \int_0^T |\widehat{K_1}(t,k)|\, |\hat{\boldsymbol{\psi}}_0(k)| \, dt
+ \sum_{k \in S_3} \int_0^T |\widehat{K_1}(t,k)|\, |\hat{\boldsymbol{\psi}}_0(k)| \, dt \\[1ex]
&\leq C \sum_{k \in S_1 \cup S_2} \int_0^T e^{-\frac{|k|^2}{8} t} |\hat{\boldsymbol{\psi}}_0(k)| \, dt
+ C \sum_{k \in S_3} \int_0^T e^{-\frac{|\tilde{b}\cdot k|^2}{|k|^2} t} |\hat{\boldsymbol{\psi}}_0(k)| \, dt \\[1ex]
&\leq C \sum_{k \in S_1 \cup S_2} |\hat{\boldsymbol{\psi}}_0(k)|
+ C \sum_{k \in S_3} \frac{|k|^2}{|\tilde{b} \cdot k|^2} \left( \int_0^T \frac{|\tilde{b} \cdot k|^2}{|k|^2} e^{-\frac{|\tilde{b} \cdot k|^2}{|k|^2} t} dt \right) |\hat{\boldsymbol{\psi}}_0(k)| \\[1ex]
&\leq C \sum_{k \in \mathbb{Z}^n \setminus \{0\}} |\hat{\boldsymbol{\psi}}_0(k)|
+ C \sum_{k \in \mathbb{Z}^n \setminus \{0\}} |k|^{2+2r} |\hat{\boldsymbol{\psi}}_0(k)| \\[1ex]
&\leq C \| \boldsymbol{\psi}_0 \|_{H^m},
\end{align*}
and
\begin{align*}
J_6
&\leq  \sum_{k \in S_1 \cup S_2} \int_0^T \int_0^t |\widehat{K_1}(t-\tau,k)|\, |\widehat{N}(\tau,k)| \, d\tau dt
+ \sum_{k \in S_3} \int_0^T \int_0^t |\widehat{K_1}(t-\tau,k)|\, |\widehat{N}(\tau,k)| \, d\tau dt \\[1ex]
&\leq C \sum_{k \in S_1 \cup S_2} \int_0^T \int_0^t e^{-\frac{|k|^2}{8}(t-\tau)} |\widehat{N}(\tau,k)| \, d\tau dt
+ C \sum_{k \in S_3} \int_0^T |k|^{2+2r} |\widehat{N}(t,k)| \, dt \\[1ex]
&\leq C \sum_{k \in S_1 \cup S_2} \int_0^T \int_0^t e^{-\frac{|k|^2}{8}(t-\tau)} |k|
\left( |\hat{u} * \hat{u}| + |\hat{b} * \hat{b}| + |\hat{u} * \hat{b}| \right)(\tau,k) \, d\tau dt \\[1ex]
&\quad + C \int_0^T \sum_{k \in S_3} |k|^{3+2r}
\left( |\hat{u} * \hat{u}| + |\hat{b} * \hat{b}| + |\hat{u} * \hat{b}| \right)(t,k) \, dt \\[1ex]
&\leq C \int_0^T \sum_{k \in \mathbb{Z}^n \setminus \{0\}} |k|^{3+2r}
\left( |\hat{u} * \hat{u}| + |\hat{b} * \hat{b}| + |\hat{u} * \hat{b}| \right)(t,k) \, dt \\[1ex]
&\leq C \sup_{t \in [0, T]} \| (u, b)(t) \|_{H^m} \left( \int_0^T \sum_k |k| |\hat{u}(t,k)| \, dt + \int_0^T \sum_k |\hat{b}(t,k)| \, dt \right).
\end{align*}
Combining the estimates for \(J_5\) through \(J_8\) in \eqref{le''} completes the proof of Proposition \ref{prop5.2}.
\end{proof}

We are now ready to prove the global existence part of Theorem \ref{thm}.
Recall the energy inequality \eqref{12.24}:
\begin{equation}\label{12.24'}
E_m^2(T) \leq C(|\tilde{b}|)\|(u_0,b_0)\|_{H^m}^2
+ C E_m^4(T)
+ C E_m^2(T) \int_0^T \|\nabla u(t)\|_{L^\infty}\,dt.
\end{equation}
To preclude finite-time blow-up and extend the solution globally, it suffices to control the last term on the right-hand side of \eqref{12.24'}, namely,
\[
\int_0^T \|\nabla u(t)\|_{L^\infty}\,dt.
\]
Using the Fourier representation, we obtain for all \(t \ge 0\),
\[
\|\nabla u(t)\|_{L^\infty}
\leq C \sum_{k\in \mathbb{Z}^n\setminus\{0\}} |k|\,|\hat{u}(t,k)|,
\qquad
\| b(t)\|_{L^\infty}
\leq C \sum_{k\in \mathbb{Z}^n\setminus\{0\}}|\hat{b}(t,k)|.
\]
Hence,
\begin{align*}
\int_0^T \big(\|\nabla u(t)\|_{L^\infty} + \|b(t)\|_{L^\infty}\big)\,dt
\leq C
\sum_{k\in \mathbb{Z}^n\setminus\{0\}} \int_0^T |k|\,|\hat{u}(t,k)|\,dt
+ C\sum_{k\in \mathbb{Z}^n\setminus\{0\}} \int_0^T |\hat{b}(t,k)|\,dt.
\end{align*}
Applying Propositions \ref{prop5.1} and \ref{prop5.2}, we deduce that
\begin{align}\label{prop-appliedy}
 &\sum_{k \in \mathbb{Z}^n \setminus \{0\}} \int_{0}^{T} |k|\left| \hat{u}(t,k) \right| \, dt
+ \sum_{k \in \mathbb{Z}^n \setminus \{0\}} \int_{0}^{T}| \hat{b}(t,k)| \, dt\leq C\|(u_0,b_0)\|_{H^m} \notag\\
&\quad
+ C_1 \sup_{t \in [0,T]} \|(u, b)(t)\|_{H^m}
\left( \sum_{k \in \mathbb{Z}^n \setminus \{0\}}\int_0^T |k|\,|\hat u(t,k)| \, dt
+ \sum_{k \in \mathbb{Z}^n \setminus \{0\}}\int_0^T|\hat b(t,k)| \, dt \right),
\end{align}
for some constant \(C_1>0\).

Taking the initial data sufficiently small as in \eqref{smallconditionq}, the last term on the right-hand side of \eqref{prop-appliedy} can be absorbed into the left-hand side via a standard bootstrap argument. Consequently,
\begin{align}\label{Linftyq}
\int_0^T \|\nabla u(t)\|_{L^\infty}\,dt
\leq \int_0^T \big(\|\nabla u(t)\|_{L^\infty} + \|b(t)\|_{L^\infty}\big)\,dt
\leq C\|(u_0,b_0)\|_{H^m}, \quad \forall\, T>0.
\end{align}
Substituting \eqref{Linftyq} into \eqref{12.24'} and invoking the smallness assumption \eqref{smallconditionq} once again, we infer that
\begin{equation}\label{12.24'y}
E_m^2(T) \leq C(|\tilde{b}|)\|(u_0,b_0)\|_{H^m}^2 + C_2 E_m^4(T),
\end{equation}
for some constant \(C_2>0\).

To close the bootstrap, we make the a priori assumption
\begin{align}\label{(9)}
E_m^2(T) \leq \frac{1}{2C_2}.
\end{align}
Then from \eqref{12.24'y} it follows that
\[
E_m^2(T) \leq C(|\tilde{b}|)E_m^2(0),
\qquad \text{where } E_m(0):=\|(u_0,b_0)\|_{H^m}.
\]
Moreover, if the initial data are sufficiently small such that
\[
E_m^2(0) \leq \frac{1}{4C_2},
\]
then \(E_m^2(T) < \frac{1}{4C_2}\), which is consistent with \eqref{(9)} and thereby closes the bootstrap argument.
This yields
\begin{align}\label{eq:1.2qw}
\sup_{t \in [0, \infty)} \lVert (u, b)(t) \rVert_{H^m}^2
+ \int_0^\infty \lVert \nabla u(t) \rVert_{H^m}^2 \,{d}t
+ \int_0^\infty \lVert b(t) \rVert_{H^{m-1-r}}^2 \, {d}t
\leq C \lVert (u_0, b_0) \rVert_{H^m}^2,
\end{align}
which implies the global-in-time existence of smooth solutions and the uniform bound \eqref{eq:1.2} in Theorem \ref{thm}.

\subsection{Temporal decay estimate}\label{4/3}
In this subsection, we establish the temporal decay estimate stated in Theorem \ref{thm}.
Let $(u, b)$ be a smooth global-in-time solution to \eqref{equation} with \( \mu = 1, \nu = 0 \).

The proof relies on a time-weighted energy method applied to the modified energy functional \(Q_s(t)\) defined in \eqref{lya}.
Following the argument in \eqref{2.33}, for any $s \in [0, m]$ with \( m > 2 + r + \frac{n}{2} \), we have
\begin{align}\label{2.333t}
\frac{1}{2}\frac{d}{dt} \lVert (u, b)(t) \rVert_{H^s}^2 +\lVert \nabla u \rVert_{H^s}^2
\leq C \left( \lVert \nabla u \rVert_{L^{\infty}} + \lVert b \rVert_{H^{m-1-r}}^2 \right) \lVert (u, b)(t) \rVert_{H^s}^2.
\end{align}
Combining \eqref{2.333t} with \eqref{2.333}, we deduce
\begin{align}\label{12.2ab}
\frac{1}{2} \frac{d}{dt} Q_s(t)
\leq\; & -a \left\lVert \nabla u \right\rVert_{H^s}^2 - \frac{1}{4} \left\lVert \Lambda^{-1} (\tilde{b} \cdot \nabla b) \right\rVert_{H^s}^2
+ \left( \frac{1}{2} + \frac{|\tilde{b}|^2}{2} \right) \left\lVert \nabla u \right\rVert_{H^s}^2 \notag\\
& + C \lVert (u, b) \rVert_{H^s}^2 \left( \left\lVert \nabla u \right\rVert_{L^{\infty}} + \lVert b \rVert_{H^{m-1-r}}^2 \right) \notag\\
\leq\; & -\frac{1}{4} \left( \left\lVert \nabla u \right\rVert_{H^s}^2 + \left\lVert \Lambda^{-1} (\tilde{b} \cdot \nabla b) \right\rVert_{H^s}^2 \right)
+ C \lVert (u, b) \rVert_{H^s}^2 \left( \left\lVert \nabla u \right\rVert_{L^{\infty}} + \lVert b \rVert_{H^{m-1-r}}^2 \right).
\end{align}
Applying Lemma \ref{2.1}, Lemma \ref{lem:zero-mean-fractional} and Poincar\'e inequality, we have
\begin{align*}
\left\lVert \Lambda^{-1-r} b \right\rVert_{\dot{H}^s}
&\leq \left\lVert \Lambda^{-1-r} b \right\rVert_{H^s}
\leq c \left\lVert \tilde{b} \cdot \nabla(\Lambda^{-1-r} b) \right\rVert_{H^{s+r}} \\
&= c \left\lVert \Lambda^{-1-r}(\tilde{b} \cdot \nabla b) \right\rVert_{H^{s+r}}
\leq c_1 \left\lVert \Lambda^{-1}(\tilde{b} \cdot \nabla b) \right\rVert_{H^s},
\end{align*}
where \( c \) denoting the constant appearing in \eqref{Diophantine}.

Moreover, by Poincar\'e inequality and \eqref{meanzero2}, we get
\[
\left\lVert u \right\rVert_{H^s} \leq \left\lVert \nabla u \right\rVert_{H^s}.
\]
Then the estimate \eqref{12.2ab} can be refined as
\begin{align}\label{6.4}
\frac{d}{dt} Q_s(t)
\leq -\frac{c^*}{4} \left( \left\lVert u \right\rVert_{H^s}^2 + \left\lVert \Lambda^{-1-r} b \right\rVert_{\dot{H}^s}^2 \right)
+C \lVert (u, b) \rVert_{H^s}^2 \left( \left\lVert \nabla u \right\rVert_{L^{\infty}}+ \lVert b \rVert_{H^{m-1-r}}^2 \right),
\end{align}
where \( c^* := \min\{1, c_1\} \).

Let \( M > 0 \) be a constant to be specified later such that \( \frac{a}{M} \leq 1 \). By Plancherel's theorem, it follows that
\begin{align*}
&\frac{a}{M}\lVert\Lambda^{s}b\rVert_{L^2}^2-\lVert\Lambda^{-1 - r}b\rVert_{\dot{H}^{s}}^2\\[1ex]
=&\sum_{|k|\neq0}\left(\frac{a}{M}|k|^{2s}-|k|^{2s-2-2r}\right)|\hat{b}(k)|^2\\[1ex]
\leq&\frac{a}{M}\sum_{\frac{a}{M}>|k|^{-2 - 2r}}|k|^{2s}|\hat{b}(k)|^2\\[1ex]
=&\left(\frac{a}{M}\right)^{\frac{m - s}{1 + r}}\sum_{\frac{a}{M}>|k|^{-2 - 2r}}\left(\frac{M}{a}\right)^{\frac{m -s}{1 + r}-1}|k|^{2s - 2m+2+2r}|k|^{2m-2-2r}|\hat{b}(k)|^2\\[1ex]
\leq&\left(\frac{a}{M}\right)^{\frac{m -s}{1 + r}}\lVert
\Lambda^{m-1 - r}b\rVert_{L^2}^2\leq\frac{C}{M^{\frac{m - s}{1 + r}}}  \|b\|^2_{H^{m-1-r}}.
\end{align*}
Then, the first term on the right-hand side of \eqref{6.4}  turns to
\begin{align}\label{6.4aab}
&- \frac{c^{*}}{4} \left( \left\lVert u \right\rVert_{H^s}^2 + \left\lVert \Lambda^{-1-r} b \right\rVert_{\dot{H}^s}^2 \right)\notag\\
= &- \frac{c^{*}}{4} \left\lVert u \right\rVert_{H^s}^2 - \frac{c^* a}{4M} \left\lVert \Lambda^s b \right\rVert_{L^2}^2
+ \frac{c^*}{4} \left( \frac{a}{M} \left\lVert \Lambda^s b \right\rVert_{L^2}^2 - \left\lVert \Lambda^{-1 - r} b \right\rVert_{\dot{H}^s}^2 \right) \notag \\
\leq& - \frac{c^* a}{4 M c_0} \left( \left\lVert u \right\rVert_{H^s}^2 + \left\lVert b \right\rVert_{H^s}^2 \right)
+\frac{C}{M^{\frac{m - s}{1 + r}}}\|b\|^2_{H^{m-1-r}} \notag \\
\leq &- \frac{c^*}{8 M c_0}Q_s(t)
+ \frac{C}{M^{\frac{m - s}{1 + r}}}\|b\|^2_{H^{m-1-r}},
\end{align}
where we have used \eqref{12.21} and \( \left\lVert b \right\rVert_{H^s} \leq c_0 \left\lVert \Lambda^s b \right\rVert_{L^2} \), which follows from the Poincar\'e inequality due to  the mean-zero condition \eqref{meanzero2}.

Substituting \eqref{6.4aab} into \eqref{6.4}, we obtain
\begin{align}\label{6.6}
\frac{d}{dt} Q_s(t)
\leq - \frac{c^*}{8 M c_0} Q_s(t)
+ \frac{C}{M^{\frac{m - s}{1 + r}}} \left\lVert b \right\rVert_{H^{m - 1 - r}}^2
+ C \left\lVert (u, b) \right\rVert_{H^s}^2\left( \left\lVert \nabla u \right\rVert_{L^{\infty}}+ \lVert b \rVert_{H^{m-1-r}}^2 \right).
\end{align}
We now take \( M = a + \frac{c^* t}{8 c_0 \frac{m - s}{1 + r}} \) and multiply both sides of \eqref{6.6} by \( M^{\frac{m - s}{1 + r}} \) to deduce
\begin{align*}
\frac{d}{dt} \left( M^{\frac{m - s}{1 + r}} Q_s(t) \right)
\leq C \left\lVert b \right\rVert_{H^{m - 1 - r}}^2
+ C\left( M^{\frac{m - s}{1 + r}} Q_s(t) \right)\left( \left\lVert \nabla u \right\rVert_{L^{\infty}}+ \lVert b \rVert_{H^{m-1-r}}^2 \right).
\end{align*}
 Using Gr\"onwall's inequality and \eqref{eq:1.2qw}, we conclude the decay estimate \eqref{finaldecay}. This completes the proof of Theorem \ref{thm}.

\end{document}